\documentclass[11pt]{article}  
\usepackage{amsmath}
\usepackage{amssymb}
\usepackage{amsfonts}
\usepackage{graphicx}
\usepackage{epsfig}
\usepackage{color}
\usepackage{epsfig,amsmath,amssymb,latexsym,color,placeins}
\usepackage{algorithmic}
\usepackage[ruled,vlined]{algorithm2e}
\usepackage{graphicx}
\usepackage{subcaption}
\usepackage{xcolor}
\usepackage{epstopdf}
\usepackage[toc,page]{appendix}
\usepackage{siunitx,array,multirow}

\def\bfc{{\bf c}}

\def\bff{{\bf f}}

\def\bfx{{\bf x}}

\def\bfz{{\bf z}}

\newtheorem{theorem}{Theorem}
\newtheorem{lemma}{Lemma}

\newtheorem{example}{Example}
\newtheorem{definition}{Definition}
\newenvironment{proof}{\begin{trivlist}\item[]{\emph{Proof.}}}
               {\hfill$\Box$\end{trivlist}}

\begin{document}
\title{A Multivariate Spline based Collocation Method for 
\\ Numerical Solution of Partial Differential Equations}
\author{Ming-Jun Lai\footnote{mjlai@uga.edu, Department of Mathematics, 
University of Georgia, Athens, GA 30602. 
This author is supported by the Simons Foundation collaboration grant \#864439.}
\and Jinsil Lee
\footnote{Jinsil.Lee@uga.edu, Department of Mathematics, University of Georgia, Athens, GA 30602}}
\date{}
\maketitle

\begin{abstract}
We propose a collocation method based on multivariate polynomial splines over triangulation or 
tetrahedralization for numerical solution of partial differential equations.  
We start with a detailed explanation of the method for the Poisson equation and 
then extend the study to the second order elliptic PDE in non-divergence form. 
We shall establish the convergence of our method and show that the numerical solution can approximate 
the exact PDE solution very well. 
Then we present a large amount of numerical experimental results to demonstrate 
the performance of the method over the 2D and 3D settings. 
In addition, we present a comparison with the existing multivariate spline methods  
in \cite{ALW06} and \cite{LW17} to show that 
the new method produces  a similar and sometimes more accurate approximation 
in a more efficient fashion.   
\end{abstract}

\section{Introduction}
In this paper, we  propose and study a new collocation method based on multivariate 
splines for numerical solution of partial differential equations over polygonal domain in 
$\mathbb{R}^d$ for $d\ge 2$.  Instead of using a second order elliptic equation in 
divergence form:
\begin{equation}
	\label{GPDE}
	\left\{
	\begin{array}{cl} -\sum_{i,j=1}^d \frac{\partial}{\partial x_i}(a^{ij}(x)\frac{\partial}{\partial 
			x_j}u)+\sum_{i=1}^d b^{i}(x) \frac{\partial}{\partial x_i}u+c^1(x)u&= f, \quad x \in 
		\Omega\subset \mathbb{R}^d, \cr 
		u&=g,  \quad \hbox{ on } \partial \Omega
	\end{array}\right. 
\end{equation}
which is often used for various finite element methods, we  discuss in this paper  
a more general form of second order elliptic PDE in non-divergence form:
\begin{equation}
	\label{GPDE2}
	\left\{
	\begin{array}{cl} \sum_{i,j=1}^d a^{ij}(x)\frac{\partial}{\partial x_i}\frac{\partial}{\partial 
			x_j}u+\sum_{i=1}^d b^{i}(x) \frac{\partial}{\partial x_i}u+c(x)u&= f, \quad x \in 
		\Omega\subset \mathbb{R}^d, \cr 
		u&=g,  \quad \hbox{ on } \partial \Omega,
	\end{array}\right. 
\end{equation}
where the PDE coefficient functions $a^{ij}(x), i, j=1, \cdots, d$ 
are in $L^\infty(\Omega)$ and  satisfy the standard elliptic condition. 
In addition, when $d\ge 2$, we shall assume the so-called Cord\'es condition, 
see (\ref{cordes}) in a later section or  see  \cite{SS13}.  

Numerical solutions to the 2nd order
PDE in the non-divergence form have been studied extensively recently. 
See some studies in \cite{SS13}, \cite{LW17}, \cite{MY17},  \cite{WW19}, \cite{S19}, and etc.. 
The method in this paper provides a new and more effective approach.  
In this paper, 
we mainly use the Sobolev space $H^2(\Omega)$. 
It is known when  $\Omega$ is convex (cf. \cite{G85}), the solution to the 
Poisson equation with zero boundary condition, i.e. $g=0$ will be in
$H^2(\Omega)$.  Recently, the researchers in \cite{GL20} showed that when $\Omega$ 
has an uniformly positive reach, the solution of (\ref{GPDE2}) with zero boundary 
condition will be in $H^2(\Omega)$. Various domains of uniformly positive reach, e.g. 
star-shaped domain and  domains with holes are shown in \cite{GL20}. See more examples in 
the next preliminary section.  Many more domains other than convex domains can have $H^2$ solution.  
For any $u\in H^2(\Omega)$, we use the standard $H^2$ norm 
\begin{equation}
	\label{H2norm}
	\|u\|_{H^2} =\|u\|_{L^2(\Omega)}+ 
	\|\nabla u\|_{L^2(\Omega)}+ \sum_{i,j=1}^d 
	\|\frac{\partial}{\partial x_i}\frac{\partial}{\partial x_j}u\|_{L^2(\Omega)}
\end{equation}
for all $u$ on $H^2(\Omega)$ and the semi-norm
\begin{equation}
	\label{H2semi}
	|u|_{H^2} = \sum_{i,j=1}^d 
	\|\frac{\partial}{\partial x_i}\frac{\partial}{\partial x_j}u\|_{L^2(\Omega)}. 
\end{equation} 

Since we will use multivariate spline functions to approximate the solution $u
\in H^2(\Omega)$,  we  use $C^r$ smooth spline 
functions with $r\ge 1$ and the degree $D$ of splines  
sufficiently large satisfying $D\ge 3r+2$ in $\mathbb{R}^2$ and 
$D\ge 6r+3$  in $\mathbb{R}^3$.  
Let $S^r_D(\triangle)$ be the spline space of degree $D$ and 
smoothness $r$ over triangulation or tetrahedralization $\triangle$ of $\Omega$.
How to use such spline functions has been explained in \cite{LW04}, 
\cite{ALW06}, \cite{S15},  and \cite{S19}, and etc.. For convenience, 
we shall give a preliminary on multivariate splines in the next section. 

We now explain our spline based collocation method. For simplicity, 
we use  the standard Poisson equation which is a special case of the PDE (\ref{GPDE2}). 
\begin{equation}
	\label{Poisson}
	\left\{
	\begin{array}{cl}
		-\Delta u & = f, \quad {\hbox{ in }} \Omega\subset \mathbb{R}^d, \cr 
		u& = g,  \quad \hbox{ on } \partial \Omega. 
	\end{array}
	\right. 
\end{equation}
When $\Omega$ has a uniform positive reach, the solution to the Poisson equation will be 
in $H^2(\Omega)$. 
We  shall use $C^r$ spline functions with $r\ge 2$ to approximate the solution $u$. 
In addition, we shall use the so-called domain points (cf. \cite{LS07} or the next section) to be
the collocation points.   Letting  $\xi_i, i=1, \cdots, N$ 
be the domain points of $\triangle$ and degree $D'>0$, where $D'$ will be different from $D$, 
our multivariate spline based collocation method is to seek a spline function $s
\in S^r_D(\triangle)$ satisfying 
\begin{equation}
	\label{Poisson2}
	\left\{
	\begin{array}{cl}
		-\Delta s(\xi_i) & = f(\xi_i), \quad \forall \xi_i\in \Omega\subset \mathbb{R}^d,\cr 
		s(\xi_i) & = g(\xi_i),  \quad \forall \xi_i\in \partial \Omega, 
	\end{array}
	\right. 
\end{equation}
where $i=1, \cdots, N$.  
It is known a multivariate spline space  is a linear vector space  
which is spanned by a set of basis functions. However,  it is difficult to construct 
locally supported basis functions in $C^r(\Omega)$ with $r\ge 1$ due to the complication of
the smoothness conditions over $\triangle$. Typically, any small perturbation of a vertex in 
$\triangle$ may change the dimension of $S^r_D(\triangle)$.  
On the other hand,  the smoothness conditions can be written 
as a system of linear equations, i.e. $H{\bfc}=0$, where {$\bfc$} 
is the coefficient vector of spline function $s\in S^r_D(\triangle)$ and 
$H$ is the matrix consisting of all smoothness 
condition across each interior edge of $\triangle$ (cf. \cite{LS07} or the next section). 
To overcome this difficulty of constructing locally supported basis spline functions, 
we will begin with a discontinuous  spline space  $s\in S^{-1}_D(\triangle)$ and then add
the smoothness conditions $H{\bf c}=0$ as constraints in addition to the constraint of boundary condition. 
One of the key ideas is to let a computer decide how
to choose ${\bf c}$ to satisfy $H{\bf c}=0$ and (\ref{Poisson2}) above simultaneously.  
Clearly, (\ref{Poisson2}) leads to  a linear system which may not have a unique solution. 
It may be an over-determined linear system if $D'> D$ or an under-determined linear system 
if $D'<D$.  Our method is to use a least squares solution if the system is overdetermined 
or a sparse solution if the system is under-determined (cf. \cite{LW21}).  

To establish the convergence of the spline based collocation solution 
as the size of $\triangle$ goes to zero, we  define a new norm $\|u\|_L$ 
on $H^2(\Omega)\cap H^1_0(\Omega)$ for the Poisson equation as follows. 
\begin{equation}
	\label{H2norm2}
	\|u\|_{L}= \|\Delta u\|_{L^2(\Omega)}.
\end{equation}
We will show that the new norm is equivalent {to} the standard norm on 
Banach space $H^2(\Omega)\cap H^1_0(\Omega)$.   That is, 
\begin{theorem}
	\label{mjlai03182021}
Suppose $\Omega\subset  \mathbb{R}^d$ be a bounded domain and the closure of $\Omega$ is 
of uniformly positive reach $r_\Omega>0$.  
Then there exist two positive constants $A$ and $B$ such that 
	\begin{equation}
		\label{equivalent}
		A \|u \|_{H^2} \le \|u \|_L \le B \| u\|_{H^2}, \quad \forall u\in H^2(\Omega)\cap H^1_0(\Omega).
	\end{equation}
\end{theorem}
See the proof of Theorem~\ref{mjlai03082021} in a later section. 
Letting $u\in H^2(\Omega)\cap H^1_0(\Omega)$
be the solution of (\ref{Poisson}) with $g=0$ and $u_s$ be the spline solution 
of (\ref{Poisson2}), we use the first inequality above to have
\begin{eqnarray*}
	A\|u-u_s\|_{H^2}\le \|u-u_s\|_L.
\end{eqnarray*}
It can be seen from (\ref{Poisson2}) that the first equation can be written as
\begin{equation}
	\label{RMSE}
	\Delta (u_s(\xi_i) - u(\xi_i))=0, i = 1, \cdots, N
\end{equation}
which is a  discretization of $\|u-u_s\|_L^2$.  
%
Let $|\triangle|$ be the size of triangulation or tetrahedralization $\triangle$.  
Since we can use a spline function to approximate $u$ if $u$ is sufficiently smooth when
the size  $|\triangle|$ goes to zero (cf. \cite{LS07}), we seek the 
minimizer $u_s$ of a minimization  to be explained in a later section. 
Then the root mean square error (RMSE)  will be small for a sufficiently large
amount of collocation points and distributed evenly when the size $|\triangle|$ of
$\triangle$ is small. Then our Theorem~\ref{mjlai03182021}
implies that $\|u-u_s\|_{H^2}$ is small. Furthermore, we will show 
\begin{equation}
	\label{newestimate}
	\|u - u_s\|_{L^2(\Omega)}\le C |\triangle|^2 \|u- u_s\|_L \hbox{ and } 
	\|\nabla(u- u_s)\|_{L^2(\Omega)} \le C|\triangle|\|u- u_s\|_L 
\end{equation}
for a positive constant $C$, 
under the assumption that $u-u_s=0$ on $\partial \Omega$. 
These will establish the multivariate spline based collocation method for the Poisson equation.  

In general, we let ${\cal L}$ be the PDE operator in (\ref{GPDE3}). 
Note that we begin with  the second order term of the PDE just for convenience. 
\begin{equation}
	\label{GPDE3}
	\left\{
	\begin{array}{cl} 
		\displaystyle \sum_{i,j=1}^d a^{ij}(x)\frac{\partial}{\partial x_i}\frac{\partial}{\partial 
			x_j}u &= f, \quad x \in \Omega\subset \mathbb{R}^d, \cr 
		u &= g,  \quad \hbox{ on } \partial \Omega,
	\end{array}
	\right. 
\end{equation}

We shall similarly define a new norm associated with the PDE (\ref{GPDE3}):  
\begin{equation}
	\label{GH2norm2}
	\|u\|_{\cal L}= \|{\cal L}(u)\|_{L^2(\Omega)}.
\end{equation}
Similarly we will show  the following.   
\begin{theorem}
	\label{mjlai03172021}
	Suppose $\Omega\subset  \mathbb{R}^d$ be a bounded domain and the closure of $\Omega$ is 
	of uniformly positive reach $r_\Omega>0$. Suppose that the second order partial differential equation in (\ref{GPDE3}) is elliptic, 
	i.e. satisfying (\ref{elliptic}) and satisfies the Cord\'es condition if $d\ge 2$. 
	There exist two positive constants $A_1$ and $B_1$ such that 
	\begin{equation}
		\label{equivalent0}
		A_1 \|u \|_{H^2} \le \|u \|_{\cal L} \le B_1 \| u\|_{H^2}, \quad \forall u\in H^2(\Omega)\cap H^1_0(\Omega).
	\end{equation}
\end{theorem}
See a proof in a section later.  Similar to the Poisson equation setting, this result 
will enable us to 
establish the convergence of the spline based collocation method for the second order 
elliptic PDE in non-divergence form. Also, we will have the improved convergence similar to 
(\ref{newestimate}).  

In addition to the major advantages of spline functions: the flexibility of
the degree, the tailorable  smoothness of splines, the property of partition of the unity
of Bernstein-B\'ezier polynomials,   
there are a few more advantages of the spline based collocation methods over the traditional 
finite element methods, discontinuous Galerkin methods, virtual element methods, and etc.. 
For example, no weak formulation of the PDE solution is required and hence, 
no numerical quadrature is needed for the computation. 
For another example, it is more flexible to deal with the discontinuity arising from the 
PDE coefficients as one may easily adjust the locations of some collocation 
points close to the both sides of discontinuous curves/surfaces. In addition, the multivariate 
spline based collocation method allows one to increase the accuracy of the approximation 
by increasing the number of collocation points which can be cheaper
than finding the solution over a uniform refinement of the underlying triangulation or 
tetrahedralization within the memory budget of a computer. Besides,  our spline 
collocation method possesses tuning parameters to control  the accuracy and the smoothness  
of the spline solution.    

We shall provide many numerical results in 2D and 3D to demonstrate how well the 
spline based collocation methods can perform. Mainly, we would like to 
show the performance of solutions under the various settings: 
(1) the PDE coefficients are smooth or not very smooth, 
(2) the PDE solutions are smooth or not very smooth, (3) the domain of interest  
may not be uniformly positive reach, even very complicated
domain such such the human head used in the numerical experiment in this paper, 
and (4) the dimension $d$ can be $2$ or $3$. 
In addition, we shall compare with the existing methods in \cite{ALW06} and \cite{LW17} 
to demonstrate that the multivariate spline based collocation method can be better in the sense 
that it is more accurate and more efficient under the assumption that 
the associated collocation matrices are generated beforehand. 
Finally, we remark that we have extended our study to the biharmonic equation, i.e. 
Navier-Stokes equations and the Monge-Amp\'ere equation. 
These will leave to a near future publication, e.g. \cite{L22}.

\section{Preliminaries on  Domains of Positive Reach and Multivariate Splines}
\subsection{Domains with positive reach} 
Let us  introduce a concept on domains of interest explained in \cite{GL20}. 
\begin{definition}
Let $K\subseteq \mathbb{R}^d$ be a non-empty set. Let $r_K$ be the supremum of the number $r$ 
such that every point in 
$$P=\{x\in \mathbb{R}^d: \text{dist}(x,K)<r\}$$
has a unique projection in $K.$ The set $K$ is said to have a positive reach if $r_K>0.$ 
\end{definition}
\begin{figure}[h]
	\centering
	\begin{tabular}{cc}
		\includegraphics[width=.4\linewidth, height=4cm]{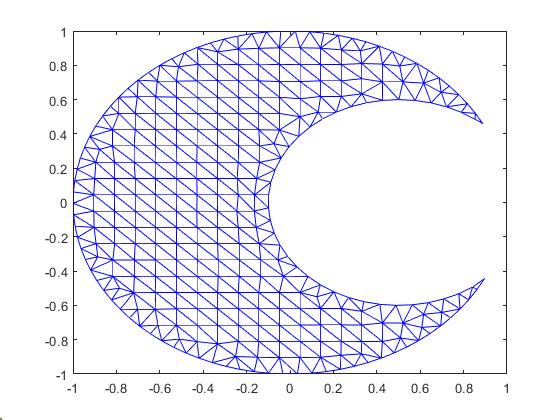}
		&
		\includegraphics[width=.4\linewidth ,height=4cm]{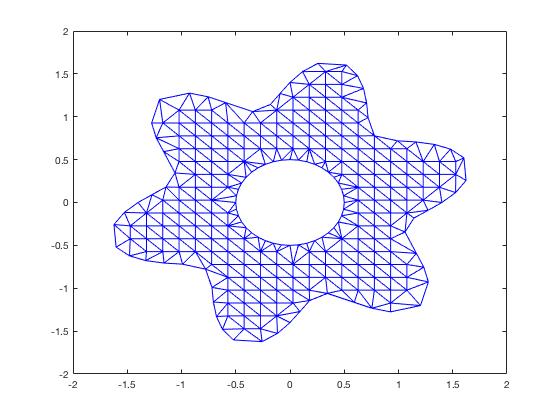}\cr 
		\includegraphics[width=.4\linewidth ,height=4cm]{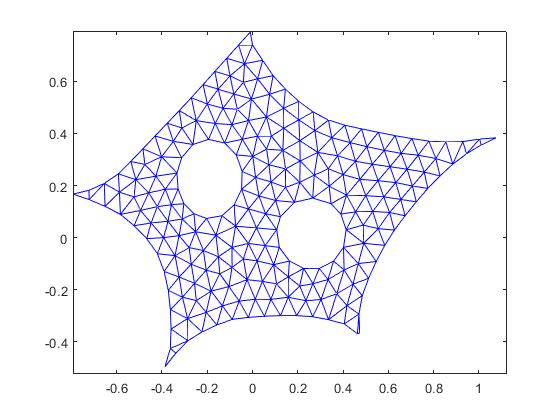}
		&
		\includegraphics[width=.4\linewidth ,height=4cm]{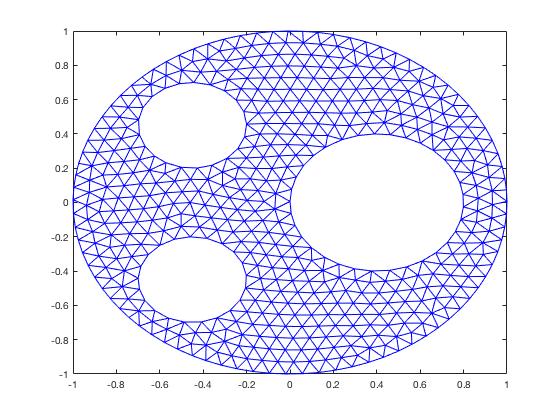}\cr 
	\end{tabular}
	\caption{Domains with positive reach}\label{positivereach}
\end{figure}
A domain with $C^2$ boundary has a positive reach. As Figure \ref{positivereach} illustrates, the domains with 
positive reach are much more general than convex domains. See  Figure~\ref{fig3D}
for domains with positive reach in the 3D setting. 
Let $B(0,\epsilon)$ be the closed ball centering at 0 with radius $\epsilon>0,$ and let $K^c$ stand for the complement 
of the set $K\in \mathbb{R}^d.$ For any $\epsilon >0,$ the set
$$E_\epsilon (K):= (K^c+B(0,\epsilon))^c \subseteq K$$
is called an $\epsilon $-erosion of $K.$
\begin{definition}
A set $K\subseteq \mathbb{R}^d$ is said to have a uniformly positive reach $r_0$ if there exists 
some $\epsilon_0>0$ such that for all $\epsilon \in [0,\epsilon_0], E_\epsilon (K)$ has a 
positive reach at least $r_0.$ 
\end{definition}
And we have the following property about these domains 
\begin{lemma}
If $\Omega \subset \mathbb{R}^d$ is of positive reach $r_0$, then for any $0<\epsilon<r_0$, the boundary of 
$\Omega_\epsilon :=\Omega+B(0,\epsilon)$ containing $\Omega$ is of $C^{1,1}$. Furthermore, $\Omega_\epsilon$ has a 
positive reach $\geq r_0-\epsilon.$
\end{lemma}
In \cite{GL20}, Gao and Lai proved the following regularity theorem which will be used to prove Theorem~\ref{mjlai03182021} in the next section. 
\begin{theorem}\label{thm:regularity} 
Let $\Omega$ be a bounded domain. Suppose the closure of $\Omega$ is of uniformly positive reach $r_\Omega$. 
For any  $f\in L^2(\Omega),$ let $u\in H^1_0(\Omega)$ be the unique weak solution of the Dirichlet problem:
	\begin{align*}
		\begin{cases}
			-\Delta u&=f~~in~\Omega\\
			u&=0~~on~\partial \Omega
		\end{cases}
	\end{align*}
	Then $u\in H^{2}(\Omega)$ in the sense that
\begin{align}\label{Poiineq}
\sum_{i,j=1}^n \int_\Omega (\frac{\partial^2 u}{\partial x_i \partial x_j})^2\le C_0\int_\Omega f^2 dx
\end{align}
for a positive constant $C_0$ depending only on $r_\Omega$.
\end{theorem}

\subsection{Multivariate Splines} 
Next we quickly summarize the essentials of multivariate splines in this subsection. 
We introduce bivariate spline functions first.  
{Before we start, we first review some facts about triangles. Given a triangle 
	$T$, we write $|T|$ for the length of its longest edge, and $\rho_T$ for the radius of the 
	largest disk that can be inscribed in $T$. }
For any polygonal domain $\Omega\subset \mathbb{R}^d$ with $d= 2$, 
let $\triangle:=\{T_1,\cdots, T_n\}$ be a triangulation of 
$\Omega$ which is a collection of triangles and $\mathcal{V}$ be the set of vertices of $\triangle$. {We called a triangulation as a quasi-uniform triangulation if all triangles $T$ in $\triangle$ have comparable sizes in the sense that $$\frac{|T|}{\rho_T}\le C<\infty, ~~~~\text{for all triangles} ~ T\in \triangle ,$$
	where $\rho_T$ is the inradius of $T$. Let $|\triangle|$ be the length of the longest edge in $\triangle.$} For a triangle 
$T=(v_1,v_2,v_3) \in \Omega,$ we define the barycentric coordinates $(b_1, b_2,b_3)$ 
of a point $(x,y)\in \Omega$. These 
coordinates are the solution to the following system of equations 
\begin{eqnarray*}
	b_1+b_2+b_3=1\\
	b_1 v_{1,x}+b_2 v_{2,x}+b_3 v_{3,x}=x\\
	b_1 v_{1,y}+b_2 v_{2,y}+b_3 v_{3,y}=y
\end{eqnarray*}
{where the vertices $v_i=(v_{i,x}, v_{i,y})$ for $i=1,2,3$} and are nonnegative if $(x,y)\in T.$ We use the barycentric coordinates to define the Bernstein polynomials of degree $D$:
\begin{eqnarray*}
	B^T_{i,j,k}(x,y):=\frac{{D}!}{i!j!k!}b_1^i b_2^j b_3^k, ~i+j+k=D,
\end{eqnarray*}
which form a basis for the space $\mathcal{P}_D$ of polynomials of degree $D$. 
Therefore, we can represent all $s\in S^{-1}_D(\triangle)$ in B-form:
\begin{eqnarray*}
	s|_T=\sum_{i+j+k=D}c_{ijk}B^T_{ijk}, \quad \forall T\in \triangle,
\end{eqnarray*}
where the B-coefficients $c_{i,j,k}$ are uniquely determined by $s$. 

Moreover, for given $T=(v_1,v_2,v_3)\in \triangle$, 
we define the associated set of domain points to be 
\begin{equation}
	\label{domainpoints}
	\mathcal{D}_{D,T}:= \{\frac{iv_{1}+jv_{2}+kv_{3}}{D} \}_{i+j+k=D}.
\end{equation}
Let $\mathcal{D}_{D,\triangle} = \cup_{T\in \triangle} \mathcal{D}_{D,T}$ be the domain 
points of triangulation $\triangle$ and degree $D$.  

We use the discontinuous spline space $S^{-1}_D(\triangle):= 
\{s|_{T} \in \mathcal{P}_D, T\in \triangle\}$ as a base. 
Then we add the smoothness conditions to define the space 
$\mathcal{S}^r_D(\triangle):=C^r(\Omega)\cap S^{-1}_D(\triangle).$ 
The smoothness conditions are explained in \cite{LS07}. They are linear equations as seen in Theorems 2.28 and 15.31 in 
\cite{LS07}.  
Let {$\bfc$} be the coefficient vector of $s\in S^{-1}_D(\triangle)$ 
and $H$ be the matrix which consists of 
the smoothness conditions across each interior edge of $\triangle$.  
Then it is known that $H{\bfc}=0$ if and only if $s\in C^r(\Omega)$ 
(cf. \cite{LS07}).   

Computations involving splines written in B-form can be performed easily according 
to \cite{LW04}, \cite{ALW06} and \cite{LW17}. 
In fact, these spline functions have numerically stable, closed-form formulas for 
differentiation, integration, and inner products. If $D\geq 3r+2$, 
spline functions on quasi-uniform triangulations have optimal approximation power. 
\begin{lemma}([Lai and Schumaker, 2007\cite{LS07}])
	\label{lem1}
	Let $k \geq 3r+2$ with $r\geq 1$. 
	Suppose $\triangle$ is a quasi-uniform triangulation of $\Omega$. 
	Then for every $u\in W_q^{k+1}(\Omega),$ there exists a quasi-interpolatory spline $s_u\in \mathcal{S}^r_k (\triangle)$ such that 
	\begin{eqnarray*}
		\|D^\alpha_x D^\beta_y (u- s_u)\|_{q, \Omega}\le C |\triangle|^{k+1-\alpha-\beta} |u|_{k+1,q,\Omega}
	\end{eqnarray*}
	for a positive constant $C$ dependent on $u, r, k$ and the smallest angle of $\triangle $, and for all $0\leq 
	\alpha+\beta \leq k$ with 
	\begin{eqnarray*}
		|u|_{k,q,\Omega}:=(\sum_{a+b=k}\|D_x^a D_y^b u\|^q_{L^q(\Omega)})^{\frac{1}{q}}.
	\end{eqnarray*}
\end{lemma}

Similarly, for trivariate splines, let $\Omega\subset\mathbb{R}^3$ and $\triangle$ 
be a tetrahedralization of $\Omega$. We 
define a trivariate spline just like bivariate splines by using 
Bernstein-{B\'ezier} polynomials defined on each tetrahedron 
$t\in \triangle$. Letting
\begin{eqnarray*}
	\mathcal{S}^r_D(\triangle)=
	\{s\in C^r(\Omega):  s|_t\in \mathbb{P}_D, t\in \triangle\} 
	=C^r(\Omega)\cap S^{-1}_D(\triangle) 
\end{eqnarray*}   
be the spline space of degree $D$ and smoothness $r\ge 0$, 
each $s\in \mathcal{S}^r_D(\triangle)$ can be rewritten as 
\begin{eqnarray*}
	s|_t= \sum_{i+j+k+\ell=D} c^t_{ijk\ell} B^t_{ijk\ell} , \quad \forall t\in \triangle, 
\end{eqnarray*}
where $B^t_{ijk\ell}$ are Bernstein-{B\'ezier} polynomials 
(cf. \cite{ALW06}, \cite{LS07}, \cite{S15} ) which are 
nonzero on $t$ and zero otherwise. Approximation properties of trivariate splines can be 
found in \cite{LS07b} and \cite{L89}.

How to use them to numerically solve partial differential equations based 
on the weak formulation like the finite element method has
been discussed in \cite{ALW06}, \cite{S15}, \cite{LW17}.  

\section{A {Spline} Based Collocation Method for the Poisson Equation}
For convenience, we simply explain our method when $d=2$ in this section. 
Numerical results in the settings of $d=2$ and $d=3$ will be given in later sections. 

For a given  triangulation $\triangle$, we use a spline space $S^r_D(\triangle)$
to
find the coefficient vector $\textbf{c}$ of spline function $\displaystyle s =\sum_{t\in 
	\triangle}\sum_{i+j+k=D}c^t_{ijk} B^t_{ijk} \in S^r_D(\triangle)$
satisfying the following equations 
\begin{equation}
	\label{PDE2}
	\left\{
	\begin{array}{cl}
		\displaystyle -  \sum_{t\in 
			\triangle}\sum_{i+j+k=D}c^t_{ijk}  \Delta B^t_{ijk}(\xi_i) 
		&= f(\xi_i), \quad \xi_i \in \Omega\subset \mathbb{R}^2 \cr 
		s(\xi_i) &=g(\xi_i),  \quad \hbox{ on } \partial \Omega,
	\end{array}
	\right.
\end{equation}
where $\{ {\xi_i}\}_{i=1,\cdots, N} \in \mathcal{D}_{D',\triangle}$  
are the domain points of $\triangle$ of degree $D'$
as explained in (\ref{domainpoints}) in the previous section and $D'>D$. 
Using these points, we have the following matrix equation:
\begin{eqnarray*}
	-K\textbf{c}:=\begin{bmatrix}-\Delta(B^t_{ijk}(\xi_i) )\end{bmatrix}
	\textbf{c}=[f(\xi_i)]=\textbf{f},
\end{eqnarray*}
where $\textbf{c}$ is the vector consisting of all spline coefficients $c^t_{ijk}, i+j+k=D, 
t\in \triangle$. In general, 
the spline $s$ with coefficients in $\bf{c}$ is a discontinuous function.    
In order to make $s\in \mathcal{S}^r_D(\triangle)$, 
its coefficient vector $\textbf{c}$ must satisfy 
the constraints $H\textbf{c}=0$ for the smoothness conditions that 
the $\mathcal{S}^r_D(\triangle)$  functions  possess (cf. \cite{LS07}).  

Based on the smoothness conditions (cf. Theorem 2.28 or Theorem 15.38 in \cite{LS07}), 
we can construct matrices $H_0$  for the $C^0$ smoothness conditions  
of spline functions and $H_r$ for the $C^r$ smoothness  conditions for $r\ge 2$, respectively.
Our collocation method is to find ${\bf c}^*$ by solving the following constrained minimization: 
\begin{align}
	\label{min1}
	\min_{\bf c} J(c)=\frac{1}{2}(\alpha \|B{\bf c}- {\bf g}\|^2+\beta 
	\|H_r{\bf c}\|^2+ \gamma \|H_0{\bf c}\|^2 )  
	\quad \text{subject to } \|K{\bf c} + {\bf f}\| \le \epsilon_1,
\end{align}
where $B, {\bf g}$ are associated with the boundary condition, $H_r$ is associated with 
the smoothness condition with $r=2$ and $H_0$ is associated with the smoothness condition with $r=0$, $\alpha>0, 
\beta>0, \gamma>0$ are fixed parameters, and $\epsilon_1>0$ is a given tolerance. It is easy to see that the 
minimization (\ref{min1}) is a convex minimization problem over a convex feasible set.  The problem (\ref{min1}) will 
have a unique solution if the feasible set is not empty.  
We shall use the following iterative method to solve the minimization problem (\ref{min1}). 
See Appendix for a derivation and a proof of
the convergence of Algorithm~\ref{alg1}. 

\begin{algorithm}\caption{Iterative Method}
	\label{alg1}
	Let $I$ be the identity matrix of $\mathbb{R}^m$. Fix $\epsilon >0.$ Given an initial guess $\lambda^{(0)} \in \text{Im}(K)$, we first compute
	\begin{align*}
		\bfz^{(1)}=(\alpha B^\top B+\beta H_r^\top H_r+\gamma H_0^\top H_0+\frac{1}{\epsilon}K^\top K)^{-1}(\alpha B^\top G+\frac{1}{\epsilon}K^\top{\bf f}-K^\top \lambda^{(0)})
	\end{align*}
	and iteratively compute
	\begin{align*}
		(\alpha B^\top B+\beta H_r^\top H_r+\gamma H_0^\top H_0+\frac{1}{\epsilon}K^\top K)
		\bfz^{(k+1)}=(\alpha B^\top B+\beta H^\top H+\gamma H_0^\top H_0)\bfz^{(k)}
		+\frac{1}{\epsilon}K^\top{\bf f}
	\end{align*}
	for $k=1,2,\cdots, $ where $\text{Im}(K)$ is the range of $K.$
\end{algorithm}

Let $u_s$ be the solution of Algorithm~\ref{alg1}.
We would like to show 
\begin{equation}
	\label{goal}
	\|u- u_s\|_{L^2(\Omega)} \le C|\triangle|^2 \epsilon_1 
\end{equation}
for some constant $C>0$, where $|\triangle|$ is the size of the underlying 
triangulation or tetrahedralization $\triangle$ 
of the domain $\Omega$.  To do so,  we first show
\begin{lemma}
	\label{lem3} 
	Suppose that $\Omega$ is a polygonal domain. 
	Suppose that $u\in H^3(\Omega)$. Then there exists a positive constant $\hat{C}$ depending 
	on $D\ge 1$ and $D'> D$ such that 
	\begin{eqnarray*}
		\|\Delta u(x,y)-\Delta u_s(x,y)\|_{L^2(\Omega)}\le \epsilon_1\hat{C}.
	\end{eqnarray*}
\end{lemma}
\begin{proof}
	Indeed, by Lemma~\ref{lem1}, we have a quasi-interpolatory spline $s_u$ satisfying 
	\begin{eqnarray*}
		| \Delta u(x,y)-\Delta s_u(x,y) |\leq 
		\epsilon, \forall (x,y)\in \Omega
	\end{eqnarray*}
for a triangulation $\triangle$ with $|\triangle|$ small enough.  Since $\Delta u(x,y)=-f(x,y)$, we have
\begin{equation}
\label{key}
\| \Delta s_u + \bff \|\leq \epsilon_1
\end{equation}
if $\epsilon$ small enough.  That is, the feasible set is not empty.  

Next we use the minimization (\ref{min1}) to have the minimizer $u_s$ satisfying  
	\begin{eqnarray*}
		|\Delta u(x_i,y_i)-\Delta u_s(x_i,y_i)|\le \epsilon_1
	\end{eqnarray*}
	with sufficiently small $|\triangle|$ for any domain points $(x_i,y_i)$ which construct the collocation matrix $K$.
	Now, these two inequalities imply that
	\begin{eqnarray*}
		| \Delta u_s(x_i,y_i)- \Delta s_u(x_i,y_i)|\leq \epsilon_1+\epsilon_1.
	\end{eqnarray*}
	Note that $\Delta u_s -\Delta s_u$ is a polynomial over each triangle $t\in \triangle$ 
	which has small values at the domain points. 
	This  implies that the polynomial $\Delta u_s - \Delta s_u$ is small over $t$.  That is, 
	\begin{align}
		\label{6}
		|\Delta u_s(x,y)-\Delta s_u(x,y)|\le C(\epsilon_1+ \epsilon_1)=2C\epsilon_1
	\end{align}
	by using Theorem 2.27 in \cite{LS07}. 
	Finally, we can use \eqref{6} to prove 
	\begin{eqnarray*}
		|\Delta u(x,y)-\Delta u_s(x,y)|=|\Delta u(x,y)-\Delta  s_u(x,y)+\Delta s_u(x,y)-\Delta  u_s(x,y)|\le \epsilon_1+ 2C\epsilon_1.
	\end{eqnarray*}
	and then $$\|\Delta u(x,y)-\Delta u_s(x,y)\|_{L^2(\Omega)}\le \epsilon_1\hat{C} $$
	for a constant $\hat{C}$ depending on the bounded domain $\Omega$ and $D, D'$, but independent of $|\triangle|$. 
\end{proof}

Now, let us consider the convergence of our method. Without loss of generality, we may 
assume $g=0$. Indeed, for any general $g$, let $u_g\in H^2(\Omega)$ be a function satisfying the
boundary condition, i.e. $u_g|_{\partial \Omega} = g$ and we consider the  Poisson equation
with solution $w=u- u_g$ and the new right-hand side $f_w= f+ \Delta u_g$.  
Recall the  standard norm on $H^2(\Omega)$ defined in (\ref{H2norm}). It is also a norm of 
$H^2(\Omega)\cap H^1_0(\Omega)$. It is easy to see that the space 
$H^2(\Omega)\cap H^1_0(\Omega)$ is a Banach space 
with the norm $\|\cdot \|_{H^2(\Omega)}$. In addition, 
let us define a new norm $\|u\|_L$ on $H^2(\Omega)\cap H^1_0(\Omega)$ as follows. 
\begin{equation}
	\label{H2normL}
	\|u\|_{L}=\|\Delta u\|_{L^2(\Omega)}
\end{equation}
We can easily show that $\|\cdot \|_L$ is a norm on $H^2(\Omega)\cap H^1_0(\Omega)$ as follows:  
Indeed, if $\|u\|_{L}=0$, then $\Delta u=0$ in $\Omega$ and $u=0$ on the boundary 
$\partial \Omega$. By the Green theorem, we get 
\begin{eqnarray*}
	\int_\Omega |\nabla u|^2=-\int_\Omega u \Delta u+\int_{\partial \Omega}u \frac{\partial u}{\partial n}=0.
\end{eqnarray*}
By Poincar\'e's inequality, we get 
$$\|u\|_{L^2(\Omega)}\le C\|\nabla u\|_{L^2(\Omega)}=0.$$ Hence, we know that $u=0.$
Next for any scalar $a$, it is trivial to have
$$\|au\|_{L}= \|\Delta (au)\|_{L^2(\Omega)} 
=|a|\|\Delta u\|_{L^2(\Omega)} =|a| \|u\|_{L}.$$ 
Finally, the triangular inequality is also trivial.
\begin{eqnarray*}
	\|u+v\|_{L}&=& \|\Delta (u+v)\|_{L^2(\Omega)} 
	\le  \|\Delta u\|_{L^2(\Omega)}+ \|\Delta v\|_{L^2(\Omega)}=\|u\|_{L}+\|v\|_{L}
\end{eqnarray*} 
by linearity of the Laplacian operator.

We now show that the new norm is equivalent to the standard norm on $H^2(\Omega)\cap H^1_0(\Omega)$. 
We are now ready to establish the following 
\begin{theorem}
	\label{mjlai03082021}
	Suppose $\Omega\subset  \mathbb{R}^d$ is a bounded domain and the closure of $\Omega$ is 
	of uniformly positive reach $r_\Omega>0$. 
	There exist two positive constants $A$ and $B$ such that 
	\begin{equation}
		\label{equivalent3}
		A \|u \|_{H^2} \le \|u \|_L \le B \| u\|_{H^2}, \quad \forall u\in H^2(\Omega)\cap H^1_0(\Omega).
	\end{equation}
\end{theorem}
\begin{proof}
We first show that $H^2(\Omega)\cap H^1_0(\Omega)$ is the Banach space with the norm $\|u\|_L$. 
Assume that $\{ u_n\}$ is the Cauchy sequence in $H^2(\Omega)\cap H^1_0(\Omega)$. We know that  $\{
\Delta u_n\}$ is a Cauchy sequence in $L^2(\Omega)$. 
Then there exists $U^*\in L^2(\Omega)$ such that $\Delta u_n$ converges to  $U^*.$  
It is known there exist a unique $u^*$ satisfying the Dirichlet problem:
	\begin{align*}
		\begin{cases}
			\Delta u&=U^*\\
			u&=0.
		\end{cases}
	\end{align*}
By Theorem~\ref{thm:regularity}, we know $u^* \in H^2(\Omega)\cap H^1_0(\Omega)$. 
Thus, we can say that there exist the unique  $u^*$  satisfying 
$\|u_n - u^*\|_L \rightarrow 0$ as $n\rightarrow \infty$. 
It is easy to get the following inequality 
	\begin{eqnarray}
		\|u\|_L &= & \|\Delta u\|_{L^2(\Omega)}  \le  \sum_{i,j=1}^d \|\frac{\partial^2}{\partial x_i\partial
			x_j}u\|_{L^2(\Omega)}\le   \|u\|_{H^2}
	\end{eqnarray}
for all $u\in H^2(\Omega)\cap H^1_0(\Omega)$.  

Next, by Theorem~\ref{thm:regularity}, more precisely, by \eqref{Poiineq}, 
we have 
$$  \|u\|_{H^2}\le C \|\Delta u\|_{L^2(\Omega)}= C \|u\|_L $$
for a constant  $C$ dependent on $C_0$ in \eqref{Poiineq}. 
Therefore, we choose $A=\frac{1}{C}$ to finish the proof. 
\end{proof}

Using Theorem~\ref{mjlai03082021}, we immediately obtain the following theorem
\begin{theorem}
	\label{mainthm2}
	Suppose $f$ and $g$ are continuous over bounded domain $\Omega \subseteq \mathbb{R}^d$ for 
	$d\ge 2$. Suppose $\Omega\subset  \mathbb{R}^d$ be a bounded domain and the closure of $\Omega$ is of uniformly positive reach $r_\Omega>0$. Suppose that $u\in H^3(\Omega)$ and $(u-u_s)|_{\partial \Omega}=0$. 
	We have the following inequality
	\begin{eqnarray*}
		\|u-u_s\|_{L^2(\Omega)} \le C\epsilon_1, \|\nabla (u-u_s)\|_{L^2(\Omega)} \le C\epsilon_1
	\end{eqnarray*}
	and 
	\begin{eqnarray*}
		\sum_{ i+j=2}\| \frac{\partial^2}{\partial x^i \partial y^j}u\|_{L^2(\Omega)}\le  C\epsilon_1
	\end{eqnarray*}
	for a positive constant $C$ depending on $A$ and $\Omega$, where $A$ is one of the constants in Theorem~\ref{mjlai03082021}. 
\end{theorem}
\begin{proof} 
	Using Lemma~\ref{lem3} and the assumption on the approximation on the boundary, we have
	\begin{eqnarray*}
		\|u-u_s\|_{H^2(\Omega)}&\le \frac{1}{A} \|\Delta (u-u_s)\|_{L^2(\Omega)}
		\le \frac{1}{A} \epsilon_1 \hat{C}.
	\end{eqnarray*}
	We choose $C =\frac{\hat{C}}{A}$ to finish the proof.  
\end{proof}
Finally we show that the convergence of $\|u- u_s\|_{L^2(\Omega)}$ and 
$\|\nabla(u- u_s)\|_{L^2(\Omega)}$ can be better.  
\begin{theorem}
	\label{mjlai05122021} Suppose that $(u-u_s)|_{\partial \Omega}=0$.  
	Under the assumptions in Theorem~\ref{mainthm2},  we have the following inequality
	\begin{eqnarray*}
		\|u-u_s\|_{L^2(\Omega)} \le C|\triangle|^2 \epsilon_1 \hbox{ and }
		\|\nabla (u-u_s)\|_{L^2(\Omega)} \le C|\triangle| \epsilon_1
	\end{eqnarray*}
	for a positive constant $C=1/A$, where $A$ is one of the constants in Theorem~\ref{mjlai03082021} and $|\triangle|$ is the
	size of the underlying triangulation $\triangle$. 
\end{theorem}
\begin{proof}
	First of all, it is known for any $w\in H^2(\Omega)\cap H^1_0(\Omega)$, there is a continuous linear spline $L_w$ over the triangulation $\triangle$
	such that 
	\begin{equation}
		\label{linespline}
		\|D_x^\alpha D_y^\beta (w- L_w)\|_{L^2(\Omega)}\le C|\triangle|^{2-\alpha-\beta} |w|_{H^2(\Omega)}
	\end{equation}
	for nonnegative integers $\alpha\ge 0, \beta\ge 0$ and $\alpha+\beta\le 2$, where $|w|_{H^2(\Omega)}$ is the semi-norm of $w$
	in $H^2(\Omega)$. Indeed, we can use the same construction method for quasi-interpolatory splines 
	used for the proof of Lemma~\ref{lem1} to establish the above estimate. 
	The above estimate will be used twice below.  
	
	By  the assumption that $u-u_s=0$ on $\partial \Omega$,  it is easy to see 
	\begin{eqnarray*}
		&&\|\nabla (u- u_s)\|^2_{L^2(\Omega)} 
		=  
		-\int_\Omega \Delta (u-u_s) (u-u_s) = -\int_\Omega \Delta (u-u_s- L_{u-u_s}) (u-u_s)\cr
		&=& \int_\Omega \nabla (u-u_s- L_{u-u_s}) \nabla (u-u_s) \le \|\nabla (u-u_s)\|_{L^2(\Omega)} 
		\| \nabla (u-u_s- L_{u-u_s})\|_{L^2(\Omega)} \cr
		&\le& \|\nabla (u-u_s)\|_{L^2(\Omega)} C|\triangle| \cdot |u-u_s|_{H^2(\Omega)}\cr
		&\le& 
		\|\nabla (u-u_s)\|_{L^2(\Omega)} |\triangle| \frac{C}{A} \|\Delta (u-u_s)\|_{L^2(\Omega)}.
	\end{eqnarray*}
	where we have used the first inequality in Theorem~\ref{mjlai03082021}.  
	It follows that $\|\nabla (u- u_s)\|_{L^2(\Omega)}\le |\triangle| \frac{C}{A}  \epsilon_1$.  
	
	Next we let $w\in H^2(\Omega)\cap H^1_0(\Omega)$ be the solution to the following Poisson equation:
	\begin{equation}
		\label{PDEw}
		\left\{
		\begin{array}{cl}
			-\Delta w&=u-u_s ~~~\text{in}~\Omega \subset\mathbb{R}^d \\
			w&=0 ~~~\text{on}~\partial \Omega,
		\end{array}
		\right.
	\end{equation}
	Then we use the continuous linear spline $L_w$ to have 
	\begin{eqnarray*}
		\| u- u_s\|^2_{L^2(\Omega)} &=& 
		-\int_\Omega \Delta w (u-u_s) = -\int_\Omega \Delta (w- L_w) (u-u_s)\cr
		&=& \int_\Omega \nabla (w-  L_w) \nabla (u-u_s) \le \|\nabla (u-u_s)\|_{L^2(\Omega)} 
		\| \nabla (w- L_w)\|_{L^2(\Omega)} \cr
		&\le& \|\nabla (u-u_s)\|_{L^2(\Omega)} C|\triangle| \cdot |w|_{H^2(\Omega)}\le 
		\frac{C}{A}|\triangle|\epsilon_1 |\triangle| \frac{C}{A} \|\Delta w\|_{L^2(\Omega)}\cr
		&=& \frac{C}{A}|\triangle|\epsilon_1 |\triangle| \frac{C}{A} \|u-u_s\|_{L^2(\Omega)}.
	\end{eqnarray*}
	where we have used the first inequality in Theorem~\ref{mjlai03082021} and the estimate of $\|\nabla(u-u_s)\|_{L^2(\Omega)}$
	above. 
	Hence, we have $\| u- u_s\|_{L^2(\Omega)}\le \frac{C^2}{A^2}|\triangle|^2 \epsilon_1$ as $|\triangle|\to 0$. 
\end{proof}

\section{General Second Order Elliptic Equations}
In this section we consider a collocation method based on bivariate/trivariate splines for a solution 
of the general second order elliptic equation in (\ref{GPDE2}).
For the PDE coefficient functions $a^{ij}, b^{i}, c^1\in L^\infty(\Omega)$, 
we assume that 
\begin{align}\label{symmetry}
	a_{ij}=a_{ji}\in L^\infty (\Omega) ~~\forall i,j =1,\cdots, d
\end{align}
and there exist $\lambda, \Lambda$ such that 
\begin{align}
	\label{elliptic}
	\lambda \sum_{i=1}^d \eta_i^2 \le \sum_{i,j}^d a^{ij}(x)\eta_i \eta_j \le \Lambda \sum_{i=1}^d \eta_i^2, \forall \eta\in \mathbb{R}^d \backslash \{0\}
\end{align}
for all $i,j$ and $x\in \Omega$. For convenience, we first assume that $b^i{= }0$ and $c^1=0$ in this section.  
In addition to the elliptic condition, we add the Cord\'es condition 
for well-posedness of the problem. We assume that there is an $\epsilon \in(0,1]$ such that
\begin{align}\label{cordes}
	\frac{\sum_{i,j=1}^d (a^{i,j})^2}{(\sum_{i=1}^d a^{ii})^2}\le \frac{1}{d-1+\epsilon} ~~a.e. ~in~ \Omega
\end{align}
Next let $\theta \in L^\infty (\Omega) $ be defined by 
\begin{eqnarray*}
	\theta:=\frac{\sum_{i=1}^d a^{ii}}{\sum_{i,j=1}^d (a^{i,j})^2}.
\end{eqnarray*}
Under these conditions, the researchers in \cite{SS13} proved the following lemma
\begin{lemma}\label{lema5}
	Let the operator $\mathcal{L}(u):=\sum_{i,j=1}^d a^{ij}(x)\frac{\partial^2}{\partial x_i\partial x_j}u$ 
	satisfy \eqref{symmetry}, \eqref{elliptic} and \eqref{cordes}. Then for any open set $U \subseteq \Omega$ and $v\in H^2(U)$, we have
	\begin{align}
		\label{SS13}
		| \theta \mathcal{L} v-\Delta v|\le \sqrt{1-\epsilon} |D^2 v| ~~a.e. ~in~U,
	\end{align}
	where $\epsilon \in (0,1]$ is as in \eqref{cordes}.
\end{lemma}
Instead of using the convexity to ensure the existence of the strong solution of (\ref{GPDE2}) 
in \cite{SS13}, we 
shall use the concept of uniformly positive reach in \cite{GL20}. 
The following is just the restatement of Theorem 3.3 in \cite{GL20}.
\begin{theorem}
	\label{GL}
	Suppose that $\Omega\subset \mathbb{R}^d$ with $d\ge 2$ is a bounded domain with uniformly positive reach.  Then the
	second order elliptic PDE in (\ref{GPDE2}) satisfying (\ref{cordes}) has a unique strong solution in $H^2(\Omega)$. 
\end{theorem}

We now extend the collocation method in the previous section to find a numerical solution of (\ref{GPDE2}). 
Similar to the discussion in the previous section, we can construct the following matrix for the PDE 
in (\ref{GPDE2}): 
\begin{eqnarray*}
	\mathcal{K}:=\begin{bmatrix}\sum_{i,j=1}^d a^{ij}(\xi_i)\frac{\partial^2}{\partial x_i\partial x_j}(B^t_{ijk}(\xi_i) )\end{bmatrix}.
\end{eqnarray*}
Similar to (\ref{min1}), consider the following minimization problem:
\begin{align}
	\label{min3}
	\min_{\bf c} J(c)=\frac{1}{2}(\alpha \|B{\bf c}- {\bf g}\|^2+  \beta \|H{\bf c}\|^2+\gamma \|H_0 {\bf c}\|^2)\quad 
	\text{subject to } -\mathcal{K} {\bf c} = {\bf f},
\end{align}

Again we will solve a nearby minimization problem as in the previous section. 
Like the Poisson equation, we let $\epsilon_1=   
\|\mathcal{K}\bfc +\bff\|$ for the minimizer $\bfc$ of (\ref{min3}). 
To study the convergence,  we may assume that $g=0$ as in the previous section so that 
the solution $u_s$ with the coefficient vector $\bfc$ which is the minimizer of (\ref{min3}) 
satisfies $u_s=0$ on $\partial\Omega$ and hence,  $\|u - u_s\|_{L^2(\partial\Omega)}=0$. Also, we have that 
$\|\mathcal {L}u_s +f\|_{L^2(\Omega)}\le \epsilon_1$. 
To show $u_s$ approximate $u$ over $\Omega$, 
let us define a new norm $\|u\|_{\mathcal{L}}$ on $H^2(\Omega)\cap H^1_0(\Omega)$ as follows. 
\begin{equation}
	\label{H2norm1}
	\|u\|_{\mathcal{L}}=\| \mathcal{L} u\|_{L^2(\Omega)}
\end{equation}
We can show that $\|\cdot \|_{\mathcal{L}}$ is a norm on $H^2(\Omega)\cap H^1_0(\Omega)$ as follows if $\epsilon \in (0, 1]$ is large enough.
Indeed, if $\|u\|_{\mathcal{L}}=0$, then $\mathcal{L} u=0$ in $\Omega$ and $u=0$ on the boundary $\partial \Omega$. 
Using this Lemma \ref{lema5} and Theorem~\ref{mjlai03082021}, we get
\begin{align}
	\int_\Omega \Delta u \Delta u-\int_\Omega (\Delta - \theta \mathcal{L})u  \Delta u  =\int_\Omega 
	\theta \mathcal{L}(u)  \Delta u = 0
\end{align}
and
\begin{eqnarray*}
	& \int_\Omega \Delta u \Delta u-\int_\Omega (\Delta-\theta \mathcal{L} )u  \Delta u  \geq \int_\Omega |\Delta u|^2-\int_\Omega  \sqrt{1-\epsilon} |D^2 u| \cdot|\Delta u|\\
	&=\int_\Omega |\Delta u|^2-\int_\Omega  \sqrt{1-\epsilon} |D^2 u| \cdot |\Delta u| 
	\geq \|\Delta u\|^2-\frac{\sqrt{1-\epsilon}}{A}\|\Delta u\| \|  \Delta u \| .
\end{eqnarray*}
Therefore, if $\epsilon >1-A^2$, then 
\begin{eqnarray*}
	(1-\frac{\sqrt{1-\epsilon}}{A})\| \Delta u\| \le 0.
\end{eqnarray*}
Hence, we know that $u=0.$ The other two properties of the norm can be proved easily.  

We mainly show that the above norm is equivalent to the standard norm on $H^2(\Omega)\cap H^1_0(\Omega)$.   Indeed, 
recall a well-known property about the norm equivalence. 
\begin{lemma}\label{brezis1}([Brezis, 2011 \cite{B11}])
	Let $E$ be a vector space equipped with two norms, $\| \cdot\|_1$ and $\|\cdot \|_2$. Assume that $E$ is a Banach space for both norms and that there exists a constant $C> 0$ such that 
	\begin{align}\label{eq26}
		\|x\|_2 \le C \| x\|_1, ~\forall x\in E.
	\end{align}
	Then the two norms are equivalent, i.e., there is a constant $c>0$ such that 
	\begin{eqnarray*}
		\|x\|_1 \le c_1 \| x\|_2, ~\forall x\in E.
	\end{eqnarray*}
\end{lemma}
\begin{proof}
	We define $E_1= (E, \|\cdot\|_1)$ and $E_2= (E, \|\cdot\|_2)$ 
	be two spaces equipped with two different norms.
	It is easy to see that $E_1$ and $E_2$ are Banach spaces. 
	Let $I$ be the identity operator which maps any u in $E_1$ to $u$ in $E_2$. Clearly, it is an injection and onto 
	because of the identity mapping and hence, it is  a surjection.
	Because of (\ref{eq26}), the mapping $I$ is a continuous operator.  
	Now we can use the well-known open mapping theorem. Let $B_1(0,1)= \{u\in E_1, \|u\|_1\le 1\}$ be an open ball. 
	The open mapping theorem says that $I (B_1(0,1))$ is open and hence, 
	it contains a ball $B_2(0, c)=\{u\in E_2, \|u\|_2 <c\}$. 
	That is, $ B_2(0, c)\subset I(B_1(0,1))$.  
	Let us claim that $c\|u\|_1 \le \| I(u)\|_2$ for all $u\in E_1$.   Otherwise, there exists 
	a $u^*$ such that  $c\|u^*\|_1> \|I(u^*)\|_2$.
	That is, $c >  \|I (u^*/\|u^*\|_1)\|_2$.  So  $I (u^*/\|u^*\|_1) \in B_2(0, c)$.  There is a $u^{**} \in B_1(0,1)$ 
	such that 
	$I u^{**} = I (u^*/\|u^*\|_1)$. Since $I$ is an injection,  $u^{**}=  I (u^*/\|u^*\|_1$.  
	Since $u^{**}\in B_1(0,1)$, 
	we have $1> \|u^{**}\|_1=  \| (u^*/\|u^*\|_1))\|=1$ which is a contradiction.   
	This shows that the claim is correct. we have
	thus $c\|u\|_1 \le \| I(u)\|_2= \|u\|_2$ for all $u\in E_1$. We choose $c_1=1/c$ to finish the proof.   
\end{proof}
Using Lemma \ref{brezis1}, we can prove the following theorem
\begin{theorem}
	\label{mjlai03222021}
	Suppose that $\Omega$ is bounded and has  uniformly positive reach $r_\Omega>0$ .
	Then there exist two positive constants $A_1$ and $B_1$ such that 
	\begin{equation}
		\label{equivalent4}
		A_1 \|u \|_{H^2(\Omega)} \le \|u \|_{\mathcal{L}} \le B_1 \| u\|_{H^2(\Omega)}, \quad \forall u\in H^2(\Omega)\cap H^1_0(\Omega).
	\end{equation}
\end{theorem}
\begin{proof}
	It follows that 
	\begin{eqnarray*}
		\|u\|_{\mathcal{L}} &\le & \max_{i,j=1\cdots , d}\|a^{ij}\|_\infty \sum_{i,j=1}^d \|\frac{\partial^2}{\partial x_i\partial
			x_j}u\|_{L^2(\Omega)}
		\le  B_1 \|u\|_{H^2(\Omega)}
	\end{eqnarray*}
	for all $u\in H^2(\Omega)\cap H^1_0(\Omega)$, where $B_1$ depending on $d,\Lambda$ and $C$.  
	Using Lemma 4 and the above inequality, there exist $\alpha_1 >0$ satisfying 
	\begin{eqnarray*}   
		\|u\|_{H^2} \le \alpha_1 \|u\|_{\mathcal{L}}.
	\end{eqnarray*}
	Therefore, we choose $A_1=\frac{1}{\alpha_1}$ to finish the proof. 
\end{proof}

\begin{theorem} 
Let $\Omega$ be a bounded and closed set satisfying the
	uniformly positive reach condition. Assume that $a^{ij}\in L^\infty (\Omega)$ satisfy \eqref{symmetry}, \eqref{elliptic} and \eqref{cordes} and 
	$\epsilon >1-A^2$. Suppose that $u\in H^3(\Omega)$ and $u-u_s =0$ on 
	$\partial \Omega$. For the solution $u$ of equation \eqref{GPDE3} and the corresponding minimizer 
	$u_s$, we have the following inequality
	\begin{eqnarray*}
		\|u-u_s\|_{L^2(\Omega)} \le C\epsilon_1
	\end{eqnarray*}
	for a positive constant $C$ depending on $\Omega$ and $A_1$ which is one of the constants in Theorem~\ref{mjlai03222021}. Similar for $\|\nabla (u-u_s)\|_{L^2(\Omega)}$ and $|u-u_s|_{H^2}$.  
\end{theorem}

Next  we consider the case that $b^i$ and $c^1$ are not zero. 
Assume that $\|a^{ij}\|_\infty$, $\|b^i \|_\infty$, $\|c^1 \|_\infty 
\le \Lambda_1$ and we denote that $\mathcal{L}_1(u):
=\sum_{i,j=1}^d a^{ij}(x)\frac{\partial^2}{\partial x_i\partial 
	x_j}u+\sum_{i=1}^d b^{i}(x) \frac{\partial}{\partial x_i}u+c^1(x)u$ and 
define a new norm $\|u\|_{\mathcal{L}_1}$ on $H^2(\Omega)\cap H^1_0(\Omega)$ as follows. 
\begin{equation}
	\label{H2norm2n}
	\|u\|_{\mathcal{L}_1}=\| \mathcal{L}_1 u\|_{L^2(\Omega)}.
\end{equation}
Assume that $\|u\|_{\mathcal{L}_1}=0$, i.e., $\mathcal{L}_1u=0$ over $\Omega$ and $u =0$ on 
$\partial \Omega$. From (\ref{SS13}), we have
\begin{eqnarray*}
	\int_\Omega \theta \mathcal{L} (u) \Delta u \geq \| \Delta u \|^2 - 
	\frac{\sqrt{1-\epsilon}}{A}\| \Delta u \|^2.
\end{eqnarray*}
Then by the above inequality we get 
\begin{eqnarray*}
	0&=&\int_\Omega \theta \mathcal{L}_1(u) \Delta u= 
	\int_\Omega \theta \mathcal{L}(u) \Delta u
	+\sum_{i=1}^d \theta  b^{i}(x) \frac{\partial}{\partial x_i}u \Delta u+\theta  c^1(x)u \Delta u\\
	&\geq&  \|\Delta u\|^2-\frac{\sqrt{1-\epsilon}}{A}\|\Delta u\|^2+
	\int_\Omega \sum_{i=1}^d \theta b^{i}(x) \frac{\partial}{\partial 
		x_i}u \Delta u+\theta  c^1(x)u\Delta u \\
	&\geq& \|\Delta u\|_{L^2(\Omega)}^2-\frac{\sqrt{1-\epsilon}}{A}\|\Delta u\|_{L^2(\Omega)}^2-\|\theta \|_\infty 
	\max_{i}\| b^i 
	\|_\infty \sqrt{d} \|\nabla u\|_{L^2(\Omega)} \|\Delta u\|_{L^2(\Omega)}\\
	&&-\|\theta \|_\infty \| c^1 \|_\infty \| u\|_{L^2(\Omega)
	} \|\Delta u\|_{L^2(\Omega)}\\
	&=& \|\Delta u\|_{L^2(\Omega)}^2-\frac{\sqrt{1-\epsilon}}{A}\|\Delta u\|_{L^2(\Omega)}^2-C_m (\|\nabla u\|_{L^2(\Omega)} \|\Delta
	u\|_{L^2(\Omega)}+\| u\|_{L^2(\Omega)} \|\Delta u\|_{L^2(\Omega)}),
\end{eqnarray*}
where $C_m=\max \{\|\theta \|_\infty \max_{i}\| b^i \|_\infty \sqrt{d}, \|\theta \|_\infty \| c^1 \|_\infty \}$. 
Dividing $\|\Delta u\|_{L^2(\Omega)}$ both sides of the inequality above 
and using Theorem \ref{mjlai03182021}, it is followed that
\begin{eqnarray*}
	0&\geq& \|\Delta u\|_{L^2(\Omega)}-\frac{\sqrt{1-\epsilon}}{A}\|\Delta u\|_{L^2(\Omega)}-C_m (\|\nabla u\|_{L^2(\Omega)} +\| u\|_{L^2(\Omega)} )\\
	&\geq& \|\Delta u\|_{L^2(\Omega)}-\frac{\sqrt{1-\epsilon}}{A}\|\Delta u\|_{L^2(\Omega)}-C_m  \| u\|_{H^2(\Omega)}\\
	&\geq& \|\Delta u\|_{L^2(\Omega)}-\frac{\sqrt{1-\epsilon}}{A}\|\Delta u\|_{L^2(\Omega)}-\frac{C_m }{A} 
	\| \Delta u\|_{L^2(\Omega)}\\
	&=&\|\Delta u\|_{L^2(\Omega)}(1-\frac{\sqrt{1-\epsilon}}{A}-\frac{C_m}{A} ).
\end{eqnarray*}
If the constant $(1-\frac{\sqrt{1-\epsilon}}{A}-\frac{C_m}{A} )$ is positive, 
then we can conclude that $\Delta u=0$. Together with  
the fact $u =0$ on $\partial \Omega$,  we know $u=0$. The other properties  $\|u+v\|_{\mathcal{L}_1} \le \|u\|_{\mathcal{L}_1}+\|v\|_{\mathcal{L}_1}$ and $\|a u\|_{\mathcal{L}_1}=|a|\|u\|_{\mathcal{L}_1}$ can be easily 
proved. The detail is omitted.

\begin{theorem}
	\label{mjlai03302021} 
	Assume that $(1-\frac{\sqrt{1-\epsilon}}{A}-\frac{C_m}{A} )>0$.
	There exist two positive constants $A_2$ and $B_2$ such that 
	\begin{equation}
		\label{equivalent5}
		A_2 \|u \|_{H^2(\Omega)} \le \|u \|_{\mathcal{L}_1} \le B_2 \| u\|_{H^2(\Omega)}, \quad \forall u\in H^2(\Omega)\cap H^1_0(\Omega).
	\end{equation}
\end{theorem}
\begin{proof}
The proof can be done by using Lemma~\ref{brezis1}.  We leave it to the interested reader.  
\end{proof}

Therefore, we can get the following theorem for the general elliptic PDE:
\begin{theorem} Suppose $\Omega\subset  \mathbb{R}^d$ be a bounded domain and the closure of $\Omega$ is 
	of uniformly positive reach $r_\Omega>0$. 
	Assume that $a^{ij}, b^i, c^1\in L^\infty (\Omega)$ 
	satisfy \eqref{symmetry}, \eqref{elliptic}, \eqref{cordes} and $(1-\frac{\sqrt{1-\epsilon}}{A}-\frac{C_m }{A} )>0$. Suppose that $u\in H^3(\Omega)$ and $u-u_s =0$ on $\partial \Omega$. 
For the solution $u$ of equation \eqref{GPDE2} and the corresponding minimizer $u_s$, we have the following inequality
\begin{eqnarray*}
		\|u-u_s\|_{L^2(\Omega)} \le C\epsilon_1
\end{eqnarray*}
	for a positive constant $C$ depending on $\Omega$ and a constant $A_2$ in Theorem~\ref{mjlai03302021}.
\end{theorem}
Finally we show that the convergence of $\|u- u_s\|_{L^2(\Omega)}$ and $\|\nabla(u- u_s)\|_{L^2(\Omega)}$ can be better 
\begin{theorem}
	\label{mjlai05182021}
	Suppose that the bounded domain $\Omega$ has an uniformly positive reach.  
	Suppose $f$ and $g$ are continuous over bounded domain $\Omega \subseteq \mathbb{R}^d$ for 
	$d=2,3$. Let $u$ be the solution of the general second order PDE \eqref{GPDE2} 
with differential operator ${\cal L}$.  
Suppose that $u\in H^3(\Omega)$. If  $u-u_s|_{\partial \Omega}=0$, we further have 
	the following inequality
	\begin{eqnarray*}
		\|u-u_s\|_{L^2(\Omega)} \le C|\triangle|^2 \epsilon_1 \hbox{ and }
		\|\nabla (u-u_s)\|_{L^2(\Omega)} \le C|\triangle| \epsilon_1
	\end{eqnarray*}
	for a positive constant $C=1/A_2$, where $A_2$ is one of the constants in Theorem~\ref{mjlai03082021} and $|\triangle|$ is the
	size of the underlying triangulation $\triangle$. 
\end{theorem}
\begin{proof}
	The proof is similar to Theorem~\ref{mjlai05122021}. We leave the detail to the interested reader. 
\end{proof}

\section{Implementation of the Spline based Collocation Method}
Before we present our computational results for Poisson equation and general second order 
elliptic equations, 
let us first explain the  implementation of our spline based collocation method. 
We divide the implementation into two parts. The first part of the implementation is to 
construct the 
collocation matrices $K$ associated with the Poisson equation and $\mathcal{K}$ 
associated with the general second order PDE in the non-divergence form over  
triangulation/tetrahedralization, based the degree $D$ of spline functions and the smoothness $r\ge 1$ 
as well as  the domain points ${\cal D}_{D', \triangle}$ 
associated with the triangulation/tetrahedralization. 
This part also generates the smoothness matrix $H_r, H_0$.  More precisely, 
for the Poisson equation, we construct collocation matrices  
\begin{equation}
	\label{part1}
	MxxV:=[(B_{ijk}^t(\bfx)_{xx}|_{\bfx=\xi_\ell}, \xi_\ell\in {\cal D}_{D', \triangle}] 
	\hbox{ and } 
	MyyV:=[(B_{ijk}^t(\bfx)_{yy}|_{\bfx=\xi_\ell}, \xi_\ell \in {\cal D}_{D', \triangle}].
\end{equation}
In fact we  choose many other points which  are in addition to the domain points to build these 
$MxxV$ and $MyyV$ to get better accuracy. For example, we choose $D'=D+3$ to 
generate domain points. Then 
$K=MxxV+MyyV$ is a size of $M\times m$ for the Poisson equation, where 
$m=\dim(S^{-1}_D(\triangle))$ and $M=\dim(S^{-1}_{D'}(\triangle))$.  
After generating matrices, we save our matrices which will be used later 
for solution of the Poisson equation for various right-hand side functions and 
various boundary conditions.  

For the general elliptic equations, we also generate all the related matrices 
$MxxV$, $MxyV$, $MyyV$, $MxV, MyV, \cdots$ similar to the matrices 
$MxxV, MyyV$ for the Poisson equation.  Then 
we  generate the collocation matrix $\mathcal{K}$ associated with the PDE coefficients  at 
the same domain points 
from all the related matrices $MxxV, MxyV, MyyV, MxV, MyV, \cdots$ 
which are just generated before. This part is the most time {consuming} step. 
See Tables~\ref{tab1} and ~\ref{tab2} for the 2D and 3D settings. 

The second part, Part 2 is to construct the right-hand side vector $\bff$ for each given 
PDE problem and the matrix $B$ and vector $G$ associated with the boundary condition
and  use Algorithm~\ref{alg1}  to solve  the minimization problem (\ref{min1}) and (\ref{min3}). 
We shall use the  four different domains in 2D shown in Fig.~\ref{positivereach} 
and four different domains in 3D shown in Fig.~\ref{fig3D} to test the performance of 
our collocation method.  
In addition, the spline based collocation method has been tested over 
many more domains of interest. In particular, many domains which may not be of positive reach  are 
used for testing and  their numerical results can be found in \cite{L22}.

\begin{figure}[h]
\centering
	\begin{tabular}{cccc}
		\includegraphics[height=2.5cm]{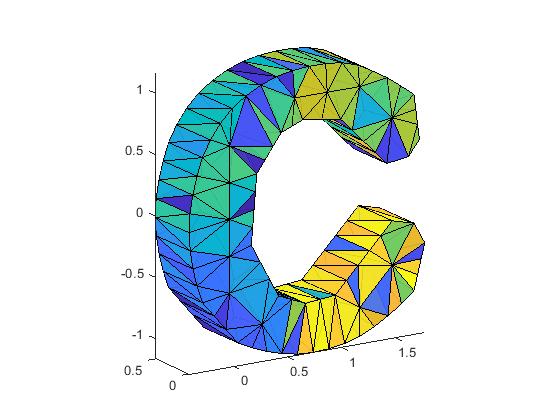}
		&\includegraphics[height=2.5cm]{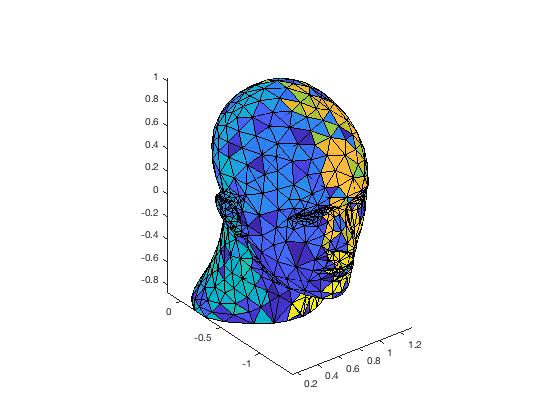}& 
		\includegraphics[height=2.5cm]{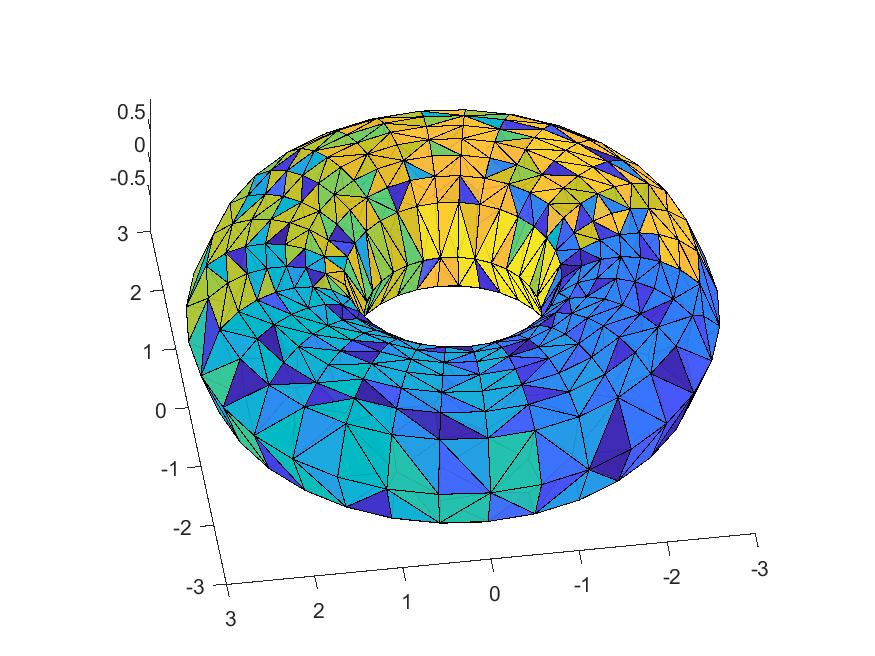} &
		\includegraphics[height=2.5cm]{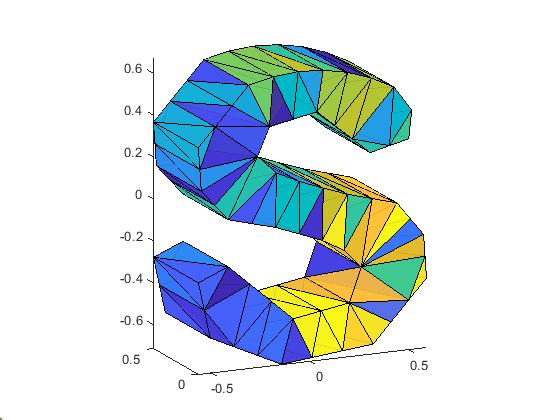}\cr 
	\end{tabular}
	\caption{Several 3D domains used for Numerical Experiments \label{fig3D}}
\end{figure}

In our computational experiments, we use a cluster computer at University of Georgia 
to generate the related collocation matrices for various degree of splines and domain points as 
described in Part I. 
We use multiple CPUs in the computer so that multiple operations 
can be done simultaneously.  
For the 2D case, we use 10 processors on a parallel computer equipped with a 12th Gen Intel(R) Core(TM) i7-12650H processor running at 2.30 GHz and 16.0 GB of installed RAM for both Part 1 and Part 2. 
And we also use a high memory (512GB) node from the Sapelo 2 cluster at University of Georgia, 
which has four AMD Opteron 6344 2.6 GHz processors. Using 48 processors on the UGA cluster, 
we can generate our necessary matrices and  
the computational times for Part 1 are listed in Table~\ref{tab1}. For 3D case, we use 48 
processors for Part 1 and 12 processors for Part 2  to do the computation.
Tables~\ref{tab2} and ~\ref{tab3} show the computational times for generating 
collocation matrices, where (P), (UGA P) indicates the time for  
the Poisson equation with 10 processors and 48 processors respectively 
and (G), (UGA G) for the general second order PDE using 10 processors 
and 48 processors, respectively.

\begin{table}
	\centering
	\scriptsize 
	\begin{tabular}{ c |c c c c c c c} 
		\hline
		Domains & Number of  & Number of & degree & Time &Time& Time &Time \\
		&vertices&triangles & &(P)&  (G)&(UGA P)&  (UGA G)\\
		\hline
Moon& 325& 531 &    8 & 0.48& 4.51    & 1.23& 1.80  \cr  
Flower& 297& 494 &    8 & 0.38& 2.47 & 0.28& 0.72      \cr  
Star& 231& 366 &    8 & 0.30& 1.49    & 0.30& 0.71  \cr  
Circle& 525& 895 &    8 & 0.85& 5.83  & 0.32& 1.86   \cr  
		\hline
	\end{tabular}
	\caption{Times in seconds for generating necessary matrices for each 2D domain in 
		Figure~\ref{positivereach}. \label{tab1}}
\end{table}

\begin{table}
	\centering
	\scriptsize 
	\begin{tabular}{ c |c c c c c c c c c} 
		\hline
		Domains & Number of  & Number of & Degree of & Time &Time &Time &Time \\
		&vertices& tetrahedron & splines &(P)&  (G)&(UGA P)&  (UGA G)\\
		\hline
			Letter C&190&431&9&5.9 & 69.0& 3.17&15.8  \cr 
Letter S&115&171&9&2.4 & 25.8&0.72&5.37 \cr 
Torus&773&2911&9&41.0 & 451.0& 8.39&82.0 \cr  
Human head&913&1588&9&21.9 & 243.3&4.53&44.7 \cr  
		\hline
	\end{tabular}
	\caption{Times in seconds for generating necessary matrices for each 3D domain in 
		Figure~\ref{fig3D}.
		\label{tab2}}
\end{table}

\begin{table}
	\centering
	\scriptsize 
	\begin{tabular}{ c |c c c c c c} 
		\hline
		Domain&Time &Time&Time &Time \\
		&(P)&  (SG)&(NSG1)&  (NSG2)\\
		\hline
		Letter C &3.71  & 6.04  & 6.07 & 5.88 \cr  
		Letter S &2.10  & 2.41  & 2.33  & 2.44 \cr  
		Torus &402.13  & 595.74 &  285.42 &  181.83 \cr   
		Human head &27.21 &  48.10&  48.96 &  48.96 \cr  
		\hline
	\end{tabular}
	\caption{Times in seconds for finding solutions of 3D Poisson equation(P), 
		general second order elliptic equation  with smooth PDE coefficients (SG) or with non-smooth 
		PDE coefficients (NSG1, NSG2) 
		for each domain in Figure~\ref{fig3D}.  \label{tab3}}
\end{table}

Another issue of our computation is how to choose $\alpha, \beta, $ and $\gamma$ 
in our Algorithm~\ref{alg1}. As there are many numerical solutions, we need to decide which 
one to choose. That is, if we 
are interested more in the accuracy of  numerical solutions  
than the smoothness of the spline  
solutions, we use $\alpha\ge 100$ while $\beta=\gamma=1$. 
On the other hand, if we are interested
more in the smoothness of spline solutions, e.g. in the computer aided geometric design, 
we use $\alpha=1$ while $\beta= \gamma\ge 10^5$ or $\beta=1$ and $\gamma=10^5$. 
Let us present Figure \ref{2derrvbeta} and \ref{2derrvalpha} to show these phenomena. 
The numerical results in Figure \ref{2derrvbeta} and \ref{2derrvalpha} are based on 
spline functions of degree $D=8$ and smoothness $r=2$ 
over one of four domains in Figure~\ref{positivereach}  for all the testing functions 
listed in the next subsection.  
In this sections, the errors are computed based on $NI=1001  \times 1001 $ equally-spaced points $\{(\eta_i)\}_{i=1}^{NI}$ fell inside the different domains.
Considered the errors have been calculated according to the norms
\begin{align*}
	\begin{cases}
		|u|_{l_2}&=\sqrt{\frac{\sum_{i=1}^{NI} (u(i))^2}{NI} }\cr
		|u|_{h_1}&=\sqrt{\frac{\sum_{i=1}^{NI} (u(i))^2+(u_x(i))^2+(u_y(i))^2}{NI} }\cr
		|u|_{l_\infty}&= \max |u(i)|,
	\end{cases}
\end{align*}
where $u(i):=u(\eta_i), u_x(i):=u_x(\eta_i), u_y(i):=u_y(\eta_i)$ for given functions $u, u_x, u_y .$
The rooted mean square(RMS) 
of vectors $ e_s=u - u_s$ 
and the maximum error of $e_s, H_0\bfc, H_r\bfc$, is computed based on those $1001^2$ equally-spaced points over the bounding box of the domain which fall into the domain.  
 When $\beta$ increases from 1 to $10^5$, the accuracies of the smoothness $|H_0 c|_{l_\infty}$ and $|H c|_{l_\infty}$ decrease, i.e., the smoothness relations can be enforced exactly. However,  
the errors $|e_s|_{l_2}$ and $|e_s|_{h}$ increase. 
Figure \ref{2derrvalpha} shows that we get the better numerical solutions when $\alpha=100>1=\beta=\gamma$
for some testing functions, but get a worse approximation for other testing functions.  Our method offers an advantage to have
a control for producing a more smooth looking, but less 
accurate numerical solution or a more accurate, but slightly 
bumpy solution. In this paper, we emphasize the accuracy of 
spline solutions when reporting our numerical results which 
can be compared with the standard FEM or DC methods.     
For the numerical experiments in the subsequent sections, we choose $\alpha=10^2, \beta=1$ and $\gamma=1$ to get the better $l_2, h_1$ errors.

\begin{figure}[htpb]
\centering
	\begin{tabular}{cc}
		\includegraphics[height=4cm]{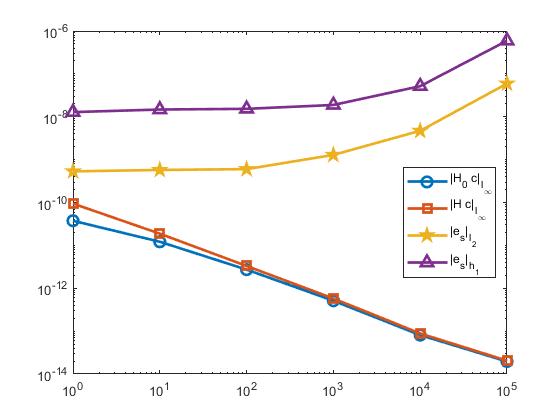}&	\includegraphics[height=4cm]{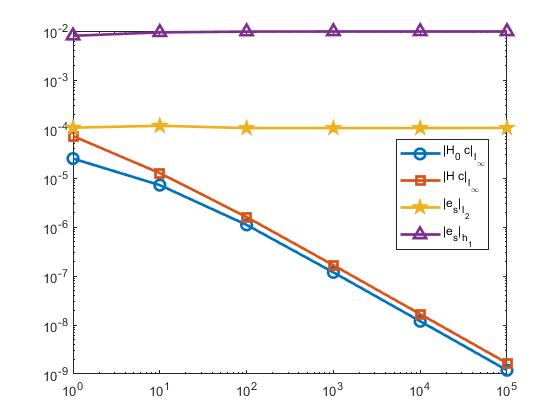}\cr
	\end{tabular}
\caption{The accuracies of the solutions $|e_s|_{l_2},|e_s|_{h}$ and the smoothness 
$|H_0 c|_{l_\infty},|H c|_{l_\infty}$  based on testing functions $ u^{s5}$(left) and 
$u^{s8}$(right) with $\alpha=\gamma=1$ for various $\beta$}\label{2derrvbeta}
\end{figure}
	\begin{figure}[htpb]
	\centering
	\begin{tabular}{cc}
		\includegraphics[height=4cm]{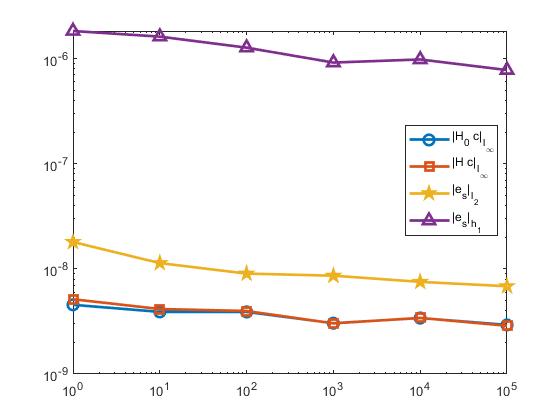}
		&\includegraphics[height=4cm]{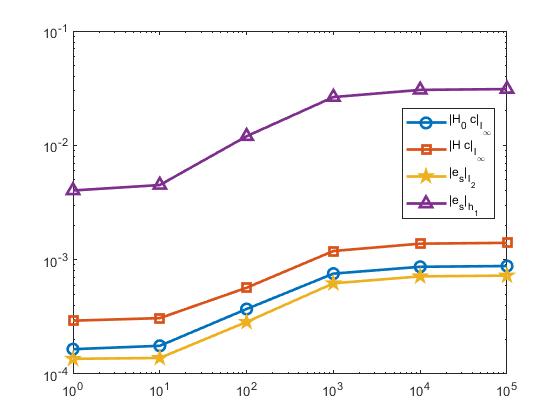}\cr
	\end{tabular}
	\caption{The accuracies of the solutions $|e_s|_{l_2},|e_s|_{h}$ and the smoothness 
$|H_0 c|_{l_\infty},|H c|_{l_\infty}$  based on testing functions $ u^{s5}$(left) and 
$u^{s8}$(right) with $\beta=\gamma=1$ for various $\alpha $}\label{2derrvalpha}
\end{figure}

\section{Numerical results for the Poisson Equation}
We shall present computational results for 2D Poisson equation and 3D Poisson equations 
separately in the following two subsections. In each section, we first present the computational
results from  the spline based collocation method to demonstrate the accuracy the method can 
achieve. Then we present a comparison of our collocation method with the numerical method proposed in 
\cite{ALW06} which uses multivariate splines to find the weak solution like finite element method.
For convenience, we shall call our spline based collocation method the LL method
and the numerical method in \cite{ALW06} the AWL method. 

\subsection{Numerical examples for 2D Poisson equations}
We have used various triangulations over various bounded domains to experiment the 
performance of our Algorithm ~\ref{alg1} in \cite{L22}  
and tested many solutions to the Poisson equation to see the accuracy that the LL method can do.
For convenience, we shall only present a few of the computational  results based on the domains
in Figure \ref{positivereach}.
The following is a list of 10 testing functions (8 smooth solutions and 2 not very smooth)
\begin{eqnarray*}
	u^{s1}&=& e^{\frac{(x^2+y^2)}{2}}, \cr
	u^{s2}&=&\cos(xy)+\cos(\pi (x^2+y^2)), \cr
	u^{s3}&=& \frac{1}{1+x^2+y^2}, \cr
	u^{s4}&=&\sin(\pi(x^2+y^2))+1,  \cr
	u^{s5}&=&\sin(3\pi x)\sin(3\pi y), \cr
	u^{s6}&=&\arctan(x^2-y^2), \cr
	u^{s7}&=&-\cos(x)\cos(y)e^{-(x-\pi)^2-(y-\pi)^2} \cr
	u^{s8}&=&\tanh(20y-20x^2)-\tanh(20x-20y^2), \cr
	u^{ns1}&=& |x^2+y^2|^{0.8} \hbox{ and }\cr 
	u^{ns2}&=&(xe^{1-|x|}-x)(ye^{1-|y|}-y).
\end{eqnarray*}
Note that the testing function in $u^{s8}$ is notoriously difficult to {solve}. 
One has to use a good adaptive triangulation method (cf. \cite{LM17}).  
The   rooted mean square (RMS) of $u-u_s$ and $\nabla (u- u_s)$ %
of approximate spline solution $u_s$ 
against the exact solution $u$ are given in Table~\ref{Poisson2D1}. 
These errors are computed based on $1001 \times 1001$ equally-spaced points of the bounding box of 
a domain in in Figure \ref{positivereach} which fell inside the domain. 
We chose collocation points to create $M\times m$ matrix $K$, where $m$ is the number of 
Bernstein basis functions (the dimension of spline space $S^{-1}_D(\triangle)$) 
and Algorithm~\ref{alg1} is used    to find the numerical solutions. 
\begin{table}[h]
	\centering
	\scriptsize 
	\begin{tabular}{ c |c c |c c| c c| c c} 
		\hline
		\multicolumn{1}{c|}{} &\multicolumn{2}{c|}{Moon}&\multicolumn{2}{c|}{Flower with a hole}
		&\multicolumn{2}{c|}{Star with 2 holes}&\multicolumn{2}{c}{Circle with 3 holes}  \\
		\hline
		Solution &$u-u_s$&$\nabla (u-u_s)$&$ u-u_s$&$\nabla (u- u_s)$&$u-u_s$ 
		&$\nabla (u-u_s)$&$u-u_s$&$\nabla (u- u_s)$\\ 
		\hline
		$u^{s1}$&6.95e-11 & 4.15e-10 & 1.23e-11 &1.54e-10&1.67e-12 &6.57e-11 & 1.63e-11 & 1.68e-10      \cr  
		$u^{s2}$&3.53e-10 & 4.81e-09 & 1.83e-11 &8.79e-10&2.46e-12 &9.77e-11 & 2.65e-11 & 2.55e-10      \cr  
		$u^{s3}$&2.58e-11 & 1.81e-10 & 6.96e-12 &9.48e-11&1.48e-12 &5.66e-11 & 8.03e-12 & 8.73e-11      \cr  
		$u^{s4}$&2.53e-10 & 3.57e-09 & 2.19e-11 &6.80e-10&1.45e-12 &8.41e-11 & 1.92e-11 & 2.00e-10      \cr  
		$u^{s5}$&6.16e-08 & 1.44e-06 & 7.73e-09 &2.57e-07&3.02e-10 &1.36e-08 & 5.33e-10 & 1.87e-08      \cr  
		$u^{s6}$&1.75e-11 & 2.86e-10 & 3.23e-12 &8.71e-11&2.97e-13 &7.23e-12 & 7.51e-12 & 7.85e-11      \cr  
		$u^{s7}$&3.07e-12 & 2.27e-11 & 1.15e-12 &1.51e-11&2.81e-13 &6.69e-12 & 1.10e-12 & 1.28e-11      \cr  
		$u^{s8}$&1.06e-03 & 9.32e-02 & 8.65e-04 &8.38e-02&4.84e-05 &3.36e-03 & 5.21e-04 & 2.09e-02      \cr  
		$u^{ns1}$&7.31e-10 & 3.68e-08 & 5.18e-06 &4.94e-04 &2.62e-06 &3.89e-04 & 1.80e-05 & 3.22e-04 \cr  
		$u^{ns2}$&3.16e-04 & 2.61e-03 & 7.39e-05 &1.51e-03 &2.75e-05 &9.76e-04 & 1.91e-05 & 6.25e-04 \cr  
		\hline
	\end{tabular}
	\caption{The RMS of errors $u- u_s$ and $\nabla u-\nabla u_s$ 
		for Poisson equations for
		four domains showed in Figure~\ref{positivereach} when  $r=2$ and $D=8$. \label{Poisson2D1}}
\end{table}

From Table~\ref{Poisson2D1}, we can see that the performance of our method is excellent. 
Next let us compare with 
the numerical method in \cite{ALW06} for the same degree, the same smoothness, and the same triangulation. The
comparison results are shown in Table~\ref{Poisson2D2}.  
One can see that both methods perform very well. Our method
can achieve a better accuracy due to the reason the 
more number of collocation points is used than 
the dimension of spline space $S^{-1}_D(\triangle)$.  

\begin{table}
	\centering
	\scriptsize 
	\begin{tabular}{ c |c c| c c |c c| c c } 
		\hline
		\multicolumn{1}{c|}{} &\multicolumn{2}{c|}{Moon}&\multicolumn{2}{c|}{Flower with a hole}&\multicolumn{2}{c|}
		{Star with 2 holes}&\multicolumn{2}{c}{Circle with 3 holes}  \\
		\hline
		Sol'n &\multicolumn{1}{c}{AWL}&\multicolumn{1}{c|}{LL}&\multicolumn{1}{c}{AWL}&\multicolumn{1}{c|}
		{LL}&\multicolumn{1}{c}{AWL}&\multicolumn{1}{c|}{LL}&\multicolumn{1}{c}{AWL}&\multicolumn{1}{c}{LL}\\
		\hline
		$u^{s1}$&1.51e-07 & 6.95e-11 & 1.14e-07 &1.23e-11&2.08e-07 &1.67e-12 & 5.22e-08 & 1.63e-11      \cr  
		$u^{s2}$&1.33e-07 & 3.53e-10 & 3.79e-07 &1.83e-11&8.93e-07 &2.46e-12 & 2.35e-08 & 2.65e-11      \cr  
		$u^{s3}$&4.94e-08 & 2.58e-11 & 8.07e-08 &6.96e-12&1.44e-07 &1.48e-12 & 1.62e-08 & 8.03e-12      \cr  
		$u^{s4}$&5.77e-07 & 2.53e-10 & 3.89e-07 &2.19e-11&4.51e-07 &1.45e-12 & 2.02e-07 & 1.92e-11      \cr  
		$u^{s5}$&1.58e-06 & 6.16e-08 & 1.43e-06 &7.73e-09&1.67e-06 &3.02e-10 & 2.65e-07 & 5.33e-10      \cr  
		$u^{s6}$&5.00e-07 & 1.75e-11 & 1.44e-07 &3.23e-12&4.03e-07 &2.97e-13 & 9.47e-08 & 7.51e-12      \cr  
		$u^{s7}$&1.99e-08 & 3.07e-12 & 2.20e-08 &1.15e-12&3.30e-08 &2.81e-13 & 4.97e-09 & 1.10e-12      \cr  
		$u^{s8}$&1.31e-03 & 1.06e-03 & 1.19e-03 &8.65e-04&1.49e-04 &4.84e-05 & 7.96e-04 & 5.21e-04      \cr  
		$u^{ns1}$&1.50e-07 & 7.31e-10 & 2.39e-04 &5.18e-06 &2.26e-05 &2.62e-06 & 1.43e-05 & 1.80e-05 \cr  
		$u^{ns2}$&1.38e-03 & 3.16e-04 & 4.55e-04 &7.39e-05 &9.87e-05 &2.75e-05 & 8.57e-05 & 1.91e-05 \cr  
		\hline
	\end{tabular}
	\caption{RMSE of spline solutions for the Poisson equation 
		over the four domains in Figure~\ref{positivereach} when  
		$r=2$ and $D=8$ for both the AWL method and the LL method. 
		\label{Poisson2D2}}
\end{table}

Finally, we summarize the computational times for both methods in Table~\ref{Poisson2D3}.  
One can see the LL method
can be more efficient if the collocation matrices are already generated.  The LL method can be 
useful for time dependent
PDE such as the heat equation. We only need to generate the collocation matrix once and use it 
repeatedly for many time step iterations.  

\begin{table}
	\centering
	\scriptsize 
	\begin{tabular}{ c |c c l l } 
		\hline
		Domain& Number of  & Number of & Average time &Average time for \\
		&vertices&triangles & for AWL method& LL method (part 2)\\
		\hline
		Moon & 325 & 531 & 9.61e-01 &6.28e-01    \cr  
		Flower with a hole& 297 & 494 & 8.05e-01 &5.39e-01    \cr  
		Star with 2 holes & 231 & 366 & 5.53e-01 &3.97e-01    \cr  
		Circle with 3 holes & 525 & 895 & 1.44e+00 &9.74e-01    \cr  
		\hline
	\end{tabular}
	\caption{The number of vertices, triangles and the averaged time for solving the 2D Poisson equation 
		for each domain in Figure~\ref{positivereach}. \label{Poisson2D3}}
\end{table}

\subsection{Numerical results for the 3D Poisson equation}
We have used our collocation method to solve the 3D Poisson equation and 
the tested 10 smooth and non-smooth solution over various domains. 
For convenience, we only show a few computational results to demonstrate that
our collocation method works very well.  More detail can be found in \cite{L22}.
Our testing  solutions are as follows:
\begin{eqnarray*} 
	u^{3ds1}&=&\sin(2x+2y)\tanh(\frac{xz}{2})\cr 
	u^{3ds2}&=& e^{\frac{x^2+y^2+z^2}{2}}\cr  
	u^{3ds3}&=& \cos(xyz)+\cos(\pi(x^2+y^2+z^2))\cr
	u^{3ds4}&=&\frac{1}{1+x^2+y^2+z^2}\cr
	u^{3ds5}&=& \sin(\pi(x^2+y^2+z^2))+1\cr
	u^{3ds6}&=& 10e^{-x^2-y^2-z^2}\cr 
	u^{3ds7}&=& \sin(2\pi x)\sin(2\pi y)\sin(2\pi z)\cr
	u^{3ds8}&=& z\tanh((-\sin(x)+y^2))\cr 
	u^{3dns1}&=& |x^2+y^2+z^2|^{0.8}\cr 
	u^{3dns2}&=& (xe^{1-|x|}-x)(ye^{1-|y|}-y)(ze^{1-|z|}-z).
\end{eqnarray*}

The rooted mean squared errors %
of approximate spline solutions against the exact solution are computed based on 
$501 \times 501\times 501$ 
equally-spaced points of the bounding box of a domain  shown in Figure \ref{fig3D} which fall into the domain. 
\begin{table}
	\centering
	\scriptsize 
	\begin{tabular}{ c |c c |c c| c c |c c} 
		\hline
		\multicolumn{1}{c|}{} &\multicolumn{2}{c|}{Letter C}&\multicolumn{2}{c|}{Letter S}&\multicolumn{2}{c|}{Torus}&\multicolumn{2}{c}{Human head}  \\
		\hline
		Solution &$u-u_s$&$\nabla (u-u_s)$&$u-u_s$&$\nabla (u- u_s)$&$u-u_s$&$\nabla (u- u_s)$
		&$u-u_s$&$\nabla (u-u_s)$\\ 
		\hline
		$u^{3ds1}$&2.31e-11 & 2.52e-10 & 3.01e-12 &4.58e-11&7.87e-11 & 1.40e-09 & 4.12e-10 &5.02e-09     \cr  
		$u^{3ds2}$&5.47e-10 & 4.86e-09 & 7.53e-12 &7.31e-11&4.52e-09 & 3.24e-08 & 1.66e-08 &1.29e-07     \cr  
		$u^{3ds3}$&5.49e-07 & 8.40e-06 & 8.87e-08 &7.80e-07&3.32e-09 & 3.21e-08 & 9.96e-06 &1.65e-04     \cr  
		$u^{3ds4}$&4.83e-09 & 5.09e-08 & 4.29e-09 &3.85e-08&2.16e-09 & 1.61e-08 & 1.13e-08 &2.21e-07     \cr  
		$u^{3ds5}$&6.49e-07 & 1.67e-05 & 1.17e-07 &9.47e-07&7.07e-09 & 5.78e-08 & 3.62e-06 &5.88e-05     \cr  
		$u^{3ds6}$&3.52e-09 & 3.99e-08 & 8.39e-10 &6.53e-09&2.03e-08 & 1.72e-07 & 6.69e-08 &6.90e-07     \cr  
		$u^{3ds7}$&9.14e-06 & 8.63e-05 & 3.20e-06 &2.44e-05&1.40e-07 & 4.75e-06 & 4.31e-05 &8.24e-04     \cr  
		$u^{3ds8}$&2.05e-08 & 2.79e-07 & 3.30e-09 &3.35e-08&1.76e-10 & 2.98e-09 & 1.90e-08 &3.94e-07     \cr  
		$u^{3dns1}$&8.80e-06 & 4.66e-04 & 3.17e-05 &1.14e-03 &2.23e-09 & 1.80e-08 & 8.28e-06 &2.18e-03\cr  
		$u^{3dns2}$&8.39e-05 & 1.20e-03 & 4.30e-05 &4.65e-04 &1.20e-04 & 2.49e-03 & 8.90e-04 &5.18e-02\cr  
		\hline
	\end{tabular}
	\caption{RMS of error vectors $u-u_s$ and $\nabla u-\nabla u_s$ for the 3D Poisson equation 
		over the four domains in Figure~\ref{fig3D} based on trivariate spline functions 
		of smoothness $r=1$ and degree $D=9$ \label{Poisson3D1}}
\end{table}

We choose collocation points to create $M\times m$ matrix $K$, 
where $m$ is  the dimension of spline space
$S^{-1}_D(\triangle)$ and apply Algorithm~\ref{alg1} to find the numerical solutions.
We tested 10 functions over the domains in Figure~\ref{fig3D}. Their 
root mean square errors are presented in Table~\ref{Poisson3D1}. 
We also compare the AWL method with LL method for the numerical solution of the 3D Poisson 
equation. See numerical results in Table~\ref{Poisson3D21} which show that the LL method is 
more accurate than the AWL method when the solutions are smooth 
and is similar to the AWL method when the solutions are not very smooth.   
\begin{table}
	\centering
	\scriptsize 
	\begin{tabular}{ c |c c| c c |c c| c c } 
		\hline
		\multicolumn{1}{c|}{} &\multicolumn{4}{c|}{Torus}&\multicolumn{4}{c}{Human head} \\
		\hline
		&\multicolumn{2}{c|}{AWL}&\multicolumn{2}{c|}{LL}&\multicolumn{2}{c|}{AWL}&\multicolumn{2}{c}{LL}\\
		\hline
		Solution &$u-u_s$&$\nabla(u-u_s)$ &$u-u_s$&$\nabla(u-u_s)$
		&$u-u_s$&$\nabla(u-u_s)$&$u-u_s$&$\nabla(u-u_s)$\\ 
		\hline
		$u^{3ds1}$&3.55e-09 & 5.74e-07 & 1.79e-10 &2.04e-09&2.83e-09 & 7.56e-07 & 5.83e-12 &6.45e-11   \cr  
		$u^{3ds2}$&2.92e-08 & 1.98e-06 & 1.14e-08 &8.50e-08 &5.21e-07 &2.72e-06 & 3.45e-10 &2.95e-09   \cr  
		$u^{3ds3}$&1.07e-07 & 8.90e-06 & 5.34e-09 &3.31e-08 &6.44e-08 & 1.21e-05 & 7.26e-10 &8.21e-09   \cr  
		$u^{3ds4}$&1.88e-08 & 1.46e-06 & 3.57e-09 &2.29e-08 &1.83e-08 & 2.72e-06 & 2.68e-10 &2.76e-09   \cr  
		$u^{3ds5}$&8.25e-08 & 5.50e-06 & 1.33e-08 &8.95e-08 &6.09e-08 & 8.43e-06 & 9.75e-10 &5.78e-09   \cr  
		$u^{3ds6}$&2.50e-07 & 1.80e-05 & 3.39e-08 &1.90e-07 &1.31e-07 & 1.35e-05 & 2.35e-09 &2.47e-08   \cr  
		$u^{3ds7}$&8.07e-08 & 5.83e-06 & 1.01e-07 &2.34e-06 &1.88e-08 & 2.72e-06 & 4.19e-08 &5.21e-07   \cr  
		$u^{3ds8}$&8.16e-09 & 7.24e-07 & 6.42e-10 &4.32e-09 &8.16e-09 & 3.41e-07 & 2.69e-11 &1.66e-10   \cr  
		$u^{3dns1}$&3.92e-08 & 2.67e-06 & 5.07e-09 &3.22e-08 &3.63e-08 & 2.67e-06 & 3.82e-06 &6.23e-04   \cr  
		$u^{3dns2}$&6.30e-04 & 2.29e-03 & 1.09e-04 &1.58e-03 &3.42e-04 & 2.49e-03 & 2.30e-04 &4.84e-03   \cr  
		\hline
	\end{tabular}
	\caption{The RMSE of spline solutions for the 3D Poisson equation
		over the two domains in Figure~\ref{fig3D} based on trivariate spline functions 
		of smoothness $r=1$ and degree $D=9$ for the AWL method and LL method. \label{Poisson3D21}}
\end{table}

\section{Numerical Results for General Second Order Elliptic PDE}
We shall present computational results for 2D and 3D general second order PDEs 
separately in the following two subsections. In each section, we first present the computational
results from  the spline based collocation method to demonstrate the accuracy the method can 
achieve. Then we present a comparison of our collocation method with the numerical method based
on \cite{LW17}. For convenience, we shall call our spline based collocation method the LL method
and the numerical method in \cite{LW17} the LW method. 

\subsection{Numerical examples for 2D general second order equations}
We have used the same triangulations over various bounded domains as shown in Figure 
\ref{positivereach}
and tested the same solutions which we used for the Poisson equation for the general second order equation 
to see the accuracy that the LL method can have. 
The root mean squared error(RMSE) $u-u_s$ and $\nabla u-\nabla u_s$
of approximate spline solutions $u_s, \nabla u_s$ against the exact solutions 
$u, \nabla u$ are given in Tables in this section. 
The RMSE are computed based on $1001 \times 1001$ equally-spaced points of the bounding box of a domain 
in Figure \ref{positivereach} which fell inside the domain.  
We chose additional collocation points to create $M\times m$ matrix $\mathcal{K}$, 
where $m, M$ are  the dimension of spline space $S^{-1}_D(\triangle)$ and 
$S^{-1}_{D'}(\triangle)$, respectively.  

\subsubsection{2D general second order equations with smooth coefficients}
\begin{example}
	\label{ex0}
	We first tested our computational method to solve 
	the 2nd order elliptic equation with smooth PDE coefficients:   
	$a_{11}= x^2 +y^2, a_{12}=\cos(xy), a_{21}=e^{xy}, 
	a_{22}=x^3 +y^2-\sin(x^2+y^2), b_1=3\cos(x)y^2, 
	b_2=e^{-x^2-y^2}, c=0$.    
	Our testing functions are 2 non-smooth solutions $u^{ns1},u^{ns2}$, and 8 smooth solutions 
	$u^{s1}$ --- $u^{s8}$ given in the previous section. 
	The RMS of error vectors $u-u_s$ and $\nabla(u-u_s)$ 
	over the four domains in Figure \ref{positivereach} 
	is presented in Table~\ref{tab7}.  
	The numerical results show that the LL method works very well.
	\begin{table}[h]
		\centering
		\scriptsize 
		\begin{tabular}{ c |c c |c c| c c| c c} 
			\hline
			\multicolumn{1}{c|}{} &\multicolumn{2}{c|}{Moon}&\multicolumn{2}{c|}{Flower with a hole}
			&\multicolumn{2}{c|}{Star with 2 holes}&\multicolumn{2}{c}{Circle with 3 holes}  \\
			\hline
			Solution &$u-u_s$&$\nabla (u- u_s)$&$u-u_s$&$\nabla (u- u_s)$&$u-u_s$&$\nabla (u- u_s)$&$u-u_s$&$\nabla (u- u_s)$\\ 
			\hline
			$u^{s1}$&3.11e-10 & 6.25e-09 & 1.63e-10 &5.62e-09&4.96e-11 &1.93e-09 & 4.05e-10 & 1.06e-08        \cr  
			$u^{s2}$&7.86e-10 & 1.51e-08 & 6.95e-10 &2.98e-08&1.33e-10 &4.04e-09 & 3.27e-10 & 1.18e-08        \cr  
			$u^{s3}$&2.90e-10 & 4.12e-09 & 1.72e-10 &4.77e-09&4.07e-11 &1.60e-09 & 1.78e-10 & 6.25e-09        \cr  
			$u^{s4}$&4.79e-10 & 1.51e-08 & 4.38e-10 &1.59e-08&5.33e-11 &2.41e-09 & 5.05e-10 & 1.26e-08        \cr  
			$u^{s5}$&5.35e-08 & 3.40e-06 & 5.90e-08 &2.58e-06&8.84e-10 &6.81e-08 & 3.04e-09 & 1.93e-07        \cr  
			$u^{s6}$&1.24e-10 & 2.52e-09 & 4.19e-11 &1.83e-09&1.11e-11 &3.56e-10 & 1.29e-10 & 2.44e-09        \cr  
			$u^{s7}$&2.65e-11 & 4.32e-10 & 3.02e-11 &1.40e-09&7.06e-12 &3.07e-10 & 5.81e-11 & 1.23e-09        \cr  
			$u^{s8}$&9.04e-03 & 2.61e-01 & 9.81e-03 &3.63e-01&2.50e-04 &1.35e-02 & 2.28e-03 & 1.29e-01        \cr  
			$u^{ns1}$&8.26e-10 & 6.54e-08 & 4.62e-06 &6.85e-04 &1.69e-06 &4.87e-04 & 5.78e-05 & 6.17e-03   \cr  
			$u^{ns2}$&2.01e-04 & 3.24e-03 & 2.97e-04 &6.88e-03 &1.27e-04 &5.80e-03 & 7.33e-05 & 2.84e-03   \cr  
			\hline
		\end{tabular}
		\caption{RMSE of spline solutions for general second order elliptic equations with smooth coefficients over the four domains in Figure~\ref{positivereach} 
			when  $r=2$ and $D=8$. \label{tab7}}
	\end{table}
\end{example}

\subsubsection{2D general second order equations with non-smooth coefficients}
\begin{example}
	\label{ex1}
	In  \cite{SS13}, the researchers experimented their numerical methods 
	for  the second order PDE as follows:
	\begin{eqnarray*}
		\sum^2_{i,j=1}(1+\delta_{ij}) \frac{x_i}{|x_i|}\frac{x_j}{|x_j|} u_{x_i x_j} =f ~~ 
		\hbox{ in } \Omega,\quad 
		u=0 ~on ~\partial \Omega,
	\end{eqnarray*}
	where $\Omega =(-1,1)^2$ and the solution $u$ is $u(x,y)=(xe^{1-|x|}-x)(y e^{1-|y|}-y)$
	which is one of our testing functions.  
	It is easy to see those coefficients satisfy the Cordes condition 
	\begin{eqnarray*}
		\frac{\sum_{i,j=1}^d (a_{i,j})^2}{(\sum_{i=1}^2 a_{ii})^2}=\frac{2^2+1+1+2^2}{(2+2)^2}=\frac{10}{16}\le \frac{1}{2-1+\epsilon}
	\end{eqnarray*}
	when $\epsilon =\frac{3}{5}$. This equation was also numerically experimented 
	in \cite{LW17} and \cite{WW19}.  
	
	Let us test our method on this 2nd order elliptic equation with non-smooth coefficients 
	for the 2 non-smooth solutions $u^{ns1},u^{ns2}$, and 8 smooth solutions $u^{s1}-u^{s8}$ 
	over the four domains used in the previous section.  
	We use bivariate splines of degree $D=8$ and smoothness $r=2$ for the experiment.
	And the  RMSE of the solutions for the  four domains in Figure \ref{positivereach} 
	are reported in Table~\ref{tab10}.  It is clear to see that our method works very well. 
	
	\begin{table}
		\scriptsize 
		\begin{tabular}{ c |c c |c c| c c| c c} 
			\hline
			\multicolumn{1}{c|}{} &\multicolumn{2}{c|}{Moon}&\multicolumn{2}{c|}{Flower with a hole}
			&\multicolumn{2}{c|}{Star with 2 holes}&\multicolumn{2}{c}{Circle with 3 holes}  \\
			\hline
			Solution &$u-u_s$&$\nabla (u- u_s)$&$u-u_s$&$\nabla (u- u_s)$&$u-u_s$&$\nabla (u- u_s)$
			&$u-u_s$&$\nabla (u- u_s)$\\ 
			\hline
			$u^{s1}$&1.36e-10 & 1.24e-09 & 8.79e-11 &1.34e-09&2.85e-11 & 3.73e-09 & 1.03e-11 &1.14e-10     \cr  
			$u^{s2}$&1.95e-10 & 2.59e-09 & 1.16e-10 &2.15e-09&2.97e-11 & 2.02e-09 & 1.66e-11 &1.73e-10     \cr  
			$u^{s3}$&5.20e-11 & 4.91e-10 & 5.21e-11 &8.99e-10&1.52e-11 & 1.07e-09 & 5.20e-12 &5.85e-11     \cr  
			$u^{s4}$&2.16e-10 & 2.46e-09 & 9.81e-11 &1.83e-09&2.68e-11 & 2.06e-09 & 1.61e-11 &1.82e-10     \cr  
			$u^{s5}$&6.26e-08 & 1.27e-06 & 1.33e-08 &3.24e-07&5.04e-10 & 2.02e-08 & 7.58e-10 &1.64e-08     \cr  
			$u^{s6}$&3.92e-11 & 4.46e-10 & 1.61e-11 &2.63e-10&4.77e-12 & 2.07e-10 & 3.25e-12 &3.81e-11     \cr  
			$u^{s7}$&3.43e-12 & 3.26e-11 & 1.26e-11 &2.00e-10&2.81e-12 & 2.45e-10 & 1.22e-12 &1.20e-11     \cr  
			$u^{s8}$&1.44e-03 & 9.95e-02 & 2.86e-03 &1.20e-01&1.11e-04 & 4.07e-03 & 1.87e-04 &1.67e-02     \cr  
			$u^{ns1}$&2.00e-09 & 5.73e-08 & 1.57e-04 &3.88e-03 &2.59e-04 & 4.30e-03 & 1.50e-05 &5.31e-04\cr  
			$u^{ns2}$&1.60e-03 & 1.62e-02 & 1.03e-03 &1.73e-02 &8.84e-04 & 1.61e-02 & 2.56e-04 &4.03e-03\cr  
			\hline
		\end{tabular}
		\caption{RMSE   $u-u_s$ and $\nabla u-\nabla u_s$ 
			for the general elliptic equation with the non-smooth coefficients 
			in Example~\ref{ex1} over the four domains in Figure~\ref{positivereach} 
			when  $r=2$ and $D=8$. \label{tab10}}
	\end{table}
\end{example}

\begin{example}
	\label{ex2}
	The second example in the paper \cite{SS13} is another second order PDE:
	\begin{eqnarray*}
		\sum^2_{i,j=1}(\delta_{ij}+ \frac{x_i x_j}{|x|^2}) u_{x_i x_j} =f ~~ in ~ \Omega,\quad u=0 ~on ~\partial \Omega,
	\end{eqnarray*}
	where $\Omega =(-1,1)^2$ and the 
	solution $u$ is $u(x,y)=|x^2+y^2|^{\frac{\alpha}{2}}$ which is on the list of our
	testing functions.  
	Then those coefficients satisfy the Cordes condition when $\epsilon =\frac{4}{5}$.

	\begin{table}[h]
		\centering
		\scriptsize 
		\begin{tabular}{ c |c c |c c| c c| c c} 
			\hline
			\multicolumn{1}{c|}{} &\multicolumn{2}{c|}{Moon}&\multicolumn{2}{c|}{Flower with a hole}
			&\multicolumn{2}{c|}{Star with 2 holes}&\multicolumn{2}{c}{Circle with 3 holes}  \\
			\hline
			Solution &$u-u_s$&$\nabla (u- u_s)$&$u-u_s$&$\nabla (u- u_s)$
			&$u-u_s$&$\nabla (u- u_s)$&$u-u_s$&$\nabla (u- u_s)$\\ 
			\hline
			$u^{s1}$&2.20e-11 & 1.81e-10 & 2.26e-11 &4.32e-10&1.04e-11 & 2.38e-09 & 6.64e-12 &7.61e-11     \cr  
			$u^{s2}$&1.64e-10 & 2.67e-09 & 2.52e-11 &1.03e-09&1.50e-11 & 2.49e-09 & 7.34e-12 &1.04e-10     \cr  
			$u^{s3}$&1.41e-11 & 1.08e-10 & 1.72e-11 &3.33e-10&9.64e-12 & 1.80e-09 & 4.03e-12 &4.67e-11     \cr  
			$u^{s4}$&1.75e-10 & 2.14e-09 & 4.16e-11 &9.44e-10&1.80e-11 & 3.92e-09 & 9.03e-12 &1.22e-10     \cr  
			$u^{s5}$&3.71e-08 & 8.57e-07 & 6.13e-09 &2.01e-07&3.78e-10 & 1.58e-08 & 5.95e-10 &1.35e-08     \cr  
			$u^{s6}$&5.70e-12 & 2.31e-10 & 6.56e-12 &1.31e-10&1.33e-12 & 2.73e-10 & 1.77e-12 &2.38e-11     \cr  
			$u^{s7}$&1.42e-12 & 1.21e-11 & 3.23e-12 &6.56e-11&1.24e-12 & 2.63e-10 & 4.61e-13 &5.59e-12     \cr  
			$u^{s8}$&1.15e-03 & 8.63e-02 & 2.11e-03 &8.77e-02&5.51e-05 & 3.41e-03 & 1.46e-04 &1.61e-02     \cr  
			$u^{ns1}$&3.58e-10 & 4.08e-08 & 1.31e-04 &3.45e-03 &2.11e-04 & 4.12e-03 & 3.26e-05 &1.01e-03\cr  
			$u^{ns2}$&1.95e-04 & 1.97e-03 & 5.78e-05 &1.35e-03 &2.73e-05 & 9.06e-04 & 1.60e-05 &5.25e-04\cr  
			\hline
		\end{tabular}
		\caption{The RMS of vectors $u-u_s, \nabla u-\nabla u_s$ for general elliptic equations 
			with non-smooth coefficients in Example~\ref{ex2} over the four domains when  $r=2$ and $D=8$. 
			\label{tab12}}
	\end{table}

	Similar to Example~\ref{ex1}, 
	we use the LL method to solve the PDE above using the 10 testing functions  
	based on bivariate splines of degree $D=8$ and  smoothness $r=2$.
	See  Table~\ref{tab12} for the RMS of  error vectors.  
\end{example}

\subsection{Comparison with Numerical Method in \cite{LW17}}
We first compare  our LL method with the LW method in \cite{LW17} when numerically solving
three PDEs given in Examples~\ref{ex0}, \ref{ex1}, and \ref{ex2}. The RMSEs from the two methods
will be reported in Table~\ref{tabllwll2}.  For simplicity, 
we only present the numerical results from the two computational methods over 
the Circle with 3 holes in Table~\ref{tabllwll2}.  
We get the similar results for other 2D domains in Figure \ref{positivereach}. 
From Table~\ref{tabllwll2}, we see that the LL method produces more accurate results.

\begin{table}[h]
	\centering
	\scriptsize 
	\begin{tabular}{ c |c c| c c |c c} 
		\hline
		
		\multicolumn{1}{c|}{} &\multicolumn{2}{c|}{PDE in Example~\ref{ex0}}&\multicolumn{2}{c|}{PDE in Example~\ref{ex1}}&\multicolumn{2}{c}{PDE in Example~\ref{ex2}} \\
		\hline
		Method&\multicolumn{1}{c}{LW}&\multicolumn{1}{c|}{LL}&\multicolumn{1}{c}{LW}&\multicolumn{1}{c|}{LL}&\multicolumn{1}{c}{LW}&\multicolumn{1}{c}{LL}\\
		\hline
		$u^{s1}$&2.01e-09 & 1.49e-10 & 2.15e-06 &1.03e-11&7.47e-09 & 6.64e-12      \cr  
		$u^{s2}$&2.22e-08 & 4.31e-11 & 2.97e-05 &1.66e-11&3.86e-08 & 7.34e-12      \cr  
		$u^{s3}$&1.70e-09 & 8.85e-11 & 4.96e-06 &5.20e-12&2.97e-09 & 4.03e-12      \cr  
		$u^{s4}$&2.29e-08 & 2.12e-10 & 6.13e-05 &1.61e-11&7.66e-08 & 9.03e-12      \cr  
		$u^{s5}$&8.24e-08 & 3.37e-09 & 4.19e-04 &7.58e-10&1.20e-06 & 5.95e-10      \cr  
		$u^{s6}$&2.63e-09 & 3.72e-11 & 4.11e-06 &3.25e-12&3.15e-09 & 1.77e-12      \cr  
		$u^{s7}$&3.06e-14 & 1.05e-11 & 2.10e-11 &1.22e-12&2.66e-14 & 4.61e-13      \cr  
		$u^{s8}$&8.50e-04 & 2.26e-03 & 2.54e-03 &1.87e-04&1.78e-04 & 1.46e-04      \cr  
		$u^{ns1}$&1.35e-05 & 4.83e-05 & 3.57e-05 &1.50e-05 &1.52e-05 & 3.26e-05 \cr  
		$u^{ns2}$&2.45e-04 & 4.56e-05 & 1.60e-04 &2.56e-04 &5.92e-05 & 1.60e-05 \cr  
		\hline
	\end{tabular}
	\caption{The RMSE of spline solutions for general elliptic equations in Example~\ref{ex0}, 
		PDE with non-smooth coefficients in Example~\ref{ex1} and in Example~\ref{ex2} 
		over the Circle with 3 holes when  
		$r=2$ and $D=8$ for the LW method and the LL method, respectively. \label{tabllwll2}}
\end{table}

Next Table~\ref{tabllwll1} shows the averaged computational time for the LL method is 
shorter than the LW method.
\begin{table}[h]
	\centering
	\scriptsize 
	\begin{tabular}{ c |c c c c } 
		\hline
		Domain& Number of  & Number of & Average time &Average time \\
		&vertices&triangles & for LW method&for Part 2 of LL method\\
		\hline
		Flower with a hole &297  & 494 & 1.3236e+02&3.521e-01\cr
		Circle with 3 holes &525 &  895 &4.4387e+03&8.313e-01\cr
		\hline
	\end{tabular}
	\caption{The number of vertices, triangles and the averaged time in seconds 
		for solving 2D general second order equations over the four domains in 
		Figure~\ref{positivereach} by the LW and LL methods. \label{tabllwll1}}
\end{table}

Combining  the computational results in Table~\ref{tabllwll1} and computational times
in Table~\ref{tabllwll2}, we conclude that the LL method is more effective and efficient than 
the  LW method.

\section{The Rate of Convergence of the LL method} 
Finally, we discuss the rate of convergence of the LL method. 
First  in Example~\ref{rates2D}, 
we conduct an experiment on the rate of convergence based on numerical solutions of 
the 2D general elliptic PDEs in Example \ref{ex0} over $[0, 1]^2$.  
The rate of convergence with respect to the size $h=|\triangle|$ of triangulation $\triangle$ 
is shown in Figure~\ref{2dratevsh}.
In addition we  show the rate of convergence with respect  
to the DOF which is presented in Figure \ref{2derrvdof}.   
Similarly in Example~\ref{3Dex1}, 
we first present the rate of convergence based on the numerical solutions of 
the 3D general elliptic PDE with smooth coefficients 
with respect to the size $h$ of triangulations in Figure~\ref{3derrvh}
and then we present the rate of convergence with respect to the DOFs. 



\begin{example}
	\label{rates2D} 
We numerically solved the general elliptic equations in Example \ref{ex0} with $D=8, r=2$ 
for testing functions $u^{s2}$ and $u^{s4}$ on different levels of the refinement to 
demonstrate the convergence behavior. 
The $L^2, H^1$ error vectors $u -u_s$ based on 
	$1001^2$ equally-spaced points over $[0, 1]^2$ with respect to the size $h=|\triangle|$ are reported in Figure~\ref{2dratevsh}. We can see that the rate of convergence is about $O(h^7)$. 
According to Theorem~\ref{mjlai05122021}, the $\|u- u_s\|_2 \le Ch^2\epsilon$. This shows that the 
numerical computation agrees with and even better than the theory we have.      
\begin{figure}[htpb]
		\centering
		\begin{tabular}{cc}
			\includegraphics[height=5cm]{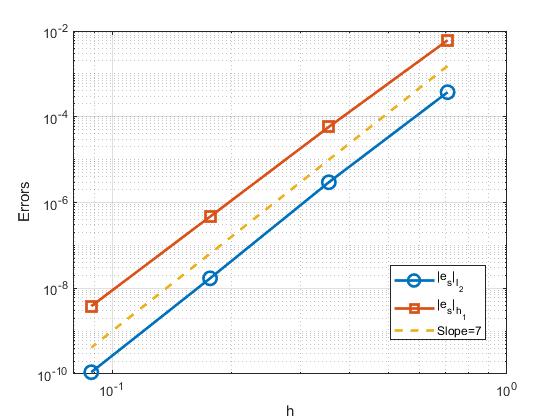}
			&\includegraphics[height=5cm]{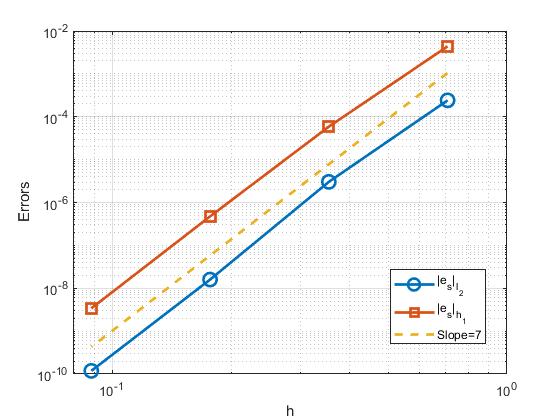}
		\end{tabular}
\caption{The RMSE in $L^2$ and $H^1$ norm of $u-u_s$ for testing functions $u^{s2}$(left) and $u^{s4}$(right) 
versus the size $h$ of
triangulation with $D=8, r=2$ where $e_s:=u-u_s$}\label{2dratevsh}
	\end{figure}
	
Next convergence results are shown in Figure \ref{2derrvdof} based on the   
DOF(=the number of triangles 
$\times \dfrac{(D+1)(D+2)}{2}$). The RMSEs  between the numerical solution and exact solutions are 
asymptotically proportional to $(\text{DOF})^{-3.5}$. 
That is, the asymptotic rate is $(\text{DOF})^{-1/(d+1)}$, where
$d=2$.   See the next example when $d=3$.  
\begin{figure}[htpb]
		\centering
		\begin{tabular}{cc}
			\includegraphics[height=5cm]{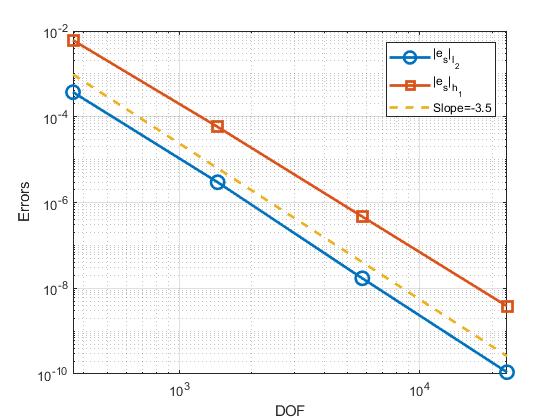}
			&\includegraphics[height=5cm]{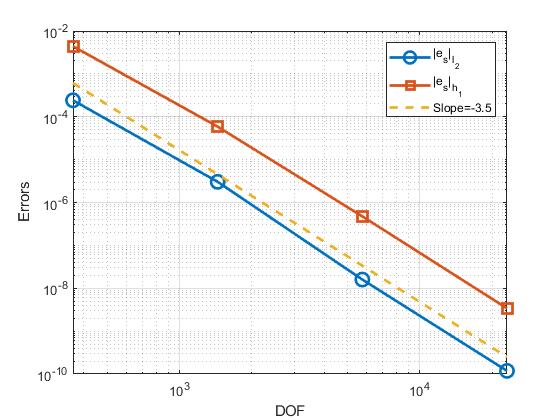}
		\end{tabular}
		\caption{The RMSE in $L^2$ and $H^1$ norm of $u-u_s$  for testing functions $u^{s2}$(left) and $u^{s4}$(right) 
versus the DOFs  with $D=8, r=2$ where $e_s:=u-u_s$ }\label{2derrvdof}
\end{figure}
\end{example}

\begin{example}
\label{3Dex1} 
We tested a 2nd order elliptic equation (\ref{GPDE2}) with smooth PDE coefficients  
	$a_{11}= x^2 +y^2, a^{22}=\cos(xy-z), a^{33}=\exp(\frac{1}{x^2+y^2+z^2+1}), 
	a^{12}+a^{21}=x^2 -y^2-z , a^{23}+a^{32}=\cos(xy-z)\sin(x-y), a^{13}+a^{31}
	= \frac{1}{y^2+z^2+1}, b_1=0, b_2=-1,b_3=\tan^{-1}(x^3-y^2+\cos(z)), c=x+y+z$, 
	where $a^{12}=a^{21}, a^{32}=a^{23}$ and $a^{13}=a^{31}.$ The testing functions are the 2 smooth solutions $u^{3ds3}, u^{3ds5}$ over the standard cube $[0,1]^3$. 
	The $L^2, H^1$ error vectors $u -u_s$ based on 
	$501^3$ equally-spaced points over $[0, 1]^3$ are reported in Figure \ref{3derrvh}. The 
errors between the numerical solution and exact solutions are asymptotically proportional to 
$\mathcal{O}(h^{7})$. We can see that the rate of convergence agrees with our theory for these smooth testing functions. 
Therefore, we conclude that the LL methods work very well.	
		\begin{figure}[htpb]
		\centering
		\begin{tabular}{cc}
			\includegraphics[height=5cm]{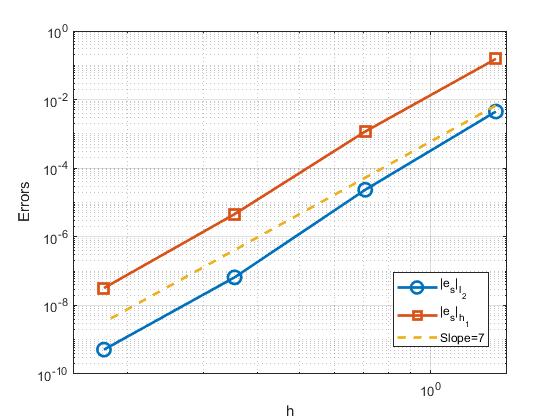}
			&\includegraphics[height=5cm]{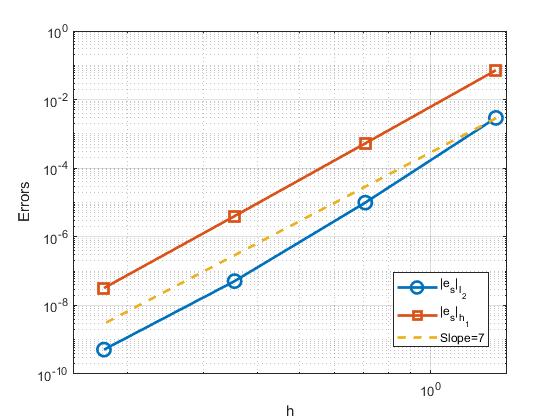}
		\end{tabular}
\caption{The RMSE in $L^2$ and $H^1$ norm of $u-u_s$  
with $D=9, r=1$ for testing functions $u^{3ds3} \text{(left)}$ and $ u^{3ds5} \text{(right)}$ 
versus the mesh size $h$}\label{3derrvh}
	\end{figure}
In addition, we show the rate of convergence with respect to the DOFs in Figure~\ref{3derrvdof} based on the   
DOF(=the number of triangles 
$\times \dfrac{(D+1)(D+2)(D+3)}{6}$). 
\begin{figure}[htpb]
		\centering
		\begin{tabular}{cc}
			\includegraphics[height=5cm]{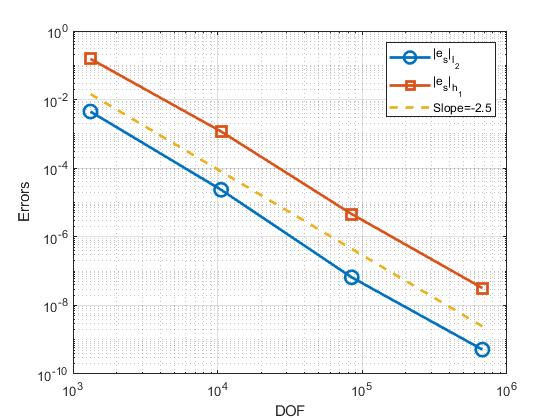}
			&\includegraphics[height=5cm]{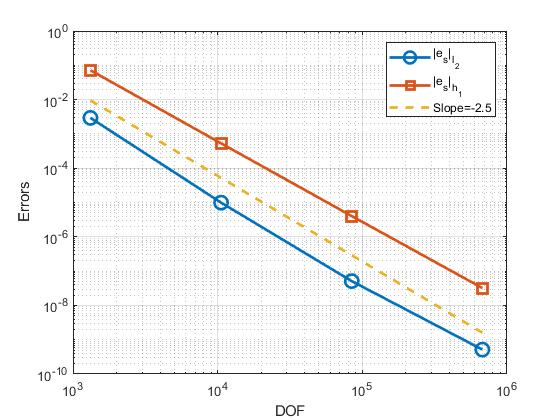}
		\end{tabular}
\caption{The RMSE in $L^2$ and $H^1$ norm of $u-u_s$  with $D=9, r=1$ for  testing functions $u^{3ds3} \text{(left)}$
and $ u^{3ds5} \text{(right)}$ versus the DOFs}\label{3derrvdof}
	\end{figure}
\end{example}

\section{Appendix:  Convergence of Algorithm~\ref{alg1}}
In this section, we first explain Algorithm~\ref{alg1} which is used 
to solve the minimization problem \eqref{min1}. In fact, Algorithm~\ref{alg1} is derived 
based on the solution to the following minimization
 \begin{align}
\label{min0}
\min_{\bf c} J(c)=\frac{1}{2}(\alpha \|B{\bf c}- {\bf g}\|^2+\beta 
\|H_r{\bf c}\|^2+ \gamma \|H_0{\bf c}\|^2 )  
\quad \text{subject to } -K{\bf c} = {\bf f},
\end{align}
where $B, {\bf g}$ are associated with the boundary condition, $H_r$ is associated with 
the smoothness condition  $\alpha>0, \beta>0$ are fixed parameters. Let us give a reason 
why we use (\ref{min0}) to replace (\ref{min1}). By Lemma~\ref{lem1}, we know spline 
functions can approximate the solution of the PDE very well when the solution $u$ is in 
$H^3(\Omega)$. When the size $|\triangle|$ is small enough, for the quasi-interpolatory 
spline $S_u$ can approximate $u$ such that $\|\Delta (u-S_u)\|_{L^2(\Omega)}\le \epsilon_1.$ 
That is, the feasible set of (\ref{min1}) is not empty. Thus, two minimization problems (\ref{min1}) 
and (\ref{min0}) are closely related to each other. Even though there is not $\bfc$ satisfying 
$-K\bfc=\bff$ exactly, a numerical computation in a computer will give a nearby solution $\bfc$ 
such that $\|K\bfc+\bff\|\le \epsilon_1$.   
We thus seek a spline solution $u_s$ satisfying (\ref{min0}).  

We use the similar technique in \cite{AL05} and \cite{ALW06}. 
For convenience, we first consider the problem
 \begin{align}
\min_{\bf c} J(c)=\frac{1}{2}(\alpha \|B{\bf z}- {\bf g}\|^2+\beta \|H{\bf z}\|^2 )  
\quad \text{subject to } -K{\bf z} = {\bf f},
\end{align}
where $B, {\bf g}$ are from the boundary condition, $H$ is from the smoothness condition.
  By the theory of Lagrange multipliers, letting
\begin{align*}
 \mathcal{U}(z,\lambda)=\frac{1}{2}(\alpha z^\intercal B^\intercal Bz- \alpha  z^\intercal B^\intercal G -\alpha  G^\intercal Bz+\alpha G^\intercal G +\beta z^\intercal H^\intercal H z)
+\lambda^\intercal (K{\bf z} + {\bf f}),   
\end{align*}
there exist $\lambda$ such that
\begin{align}
    \frac{\partial \mathcal{U}}{\partial z}=\alpha B^\intercal Bz-\alpha B^\intercal G +\beta H^\intercal Hz+K^\intercal \lambda =0\\
        \frac{\partial \mathcal{U}}{\partial \lambda}=K{\bf z} + {\bf f}=0
\end{align}
We can rewrite above linear equations as follow:
    \begin{align}
    \label{min10}
   \begin{bmatrix}K^\intercal &\alpha B^\intercal B+\beta H^\intercal H \\ O &K\end{bmatrix} \begin{bmatrix}\lambda \\z\end{bmatrix}=\begin{bmatrix} \alpha B^\intercal G \\ -{\bf f}
   \end{bmatrix}
\end{align}
To solve (\ref{min10}), 
we consider the following sequence of problems for a fixed $\epsilon>0$:
   \begin{align}
    \label{min11}
   \begin{bmatrix}K^\intercal &\alpha B^\intercal B+\beta H^\intercal H \\ -\epsilon I &K\end{bmatrix} \begin{bmatrix}\lambda^{(k+1)} \\z^{(k+1)}\end{bmatrix}=\begin{bmatrix}\alpha  B^\intercal G \\ -{\bf f}-\epsilon \lambda^{(k)}
   \end{bmatrix}
\end{align}
for $k=0,1,\cdots,$ with an initial guess $\lambda^{(0)}=0$.
  Note that \eqref{min11} reads 
    \begin{align*}
        (\alpha B^\intercal B+\beta H^\intercal H)z^{(k+1)}+K^\intercal \lambda^{(k+1)}=\alpha B^\intercal G\\
   K z^{(k+1)} -\epsilon \lambda^{(k+1)} =-{\bf f}-\epsilon \lambda^{(k)}
    \end{align*}
Multiplying on the both sides of the second equation in \eqref{min11} by $K^\intercal$, we get 
    $$K^\intercal K z^{(k+1)} -\epsilon K^\intercal \lambda^{(k+1)}=-K^\intercal{\bf f}-\epsilon K^\intercal \lambda^{(k)}$$
    or $$K^\intercal \lambda^{(k+1)}=\frac{1}{\epsilon}K^\intercal K z^{(k+1)}-\frac{1}{\epsilon}K^\intercal{\bf f}+K^\intercal \lambda^{(k)}$$ and substitute it into the first equation in \eqref{min11} to get
    \begin{align*}
        (\alpha B^\intercal B+\beta H^\intercal H)z^{(k+1)}+\frac{1}{\epsilon}K^\intercal K z^{(k+1)}-\frac{1}{\epsilon}K^\intercal{\bf f}+K^\intercal \lambda^{(k)}=\alpha B^\intercal G.
    \end{align*}
 Simplifying the above equation leads to 
     \begin{align}
   \label{eq12}
    (\alpha B^\intercal B+\beta H^\intercal H+\frac{1}{\epsilon}K^\intercal K)z^{(k+1)}=\alpha B^\intercal G+\frac{1}{\epsilon}K^\intercal{\bf f}-K^\intercal \lambda^{(k)}
   \end{align}
   It follows that
     \begin{align}\label{initial}
    z^{(1)}=(\alpha B^\intercal B+\beta H^\intercal H+\frac{1}{\epsilon}K^\intercal K)^{-1}(\alpha B^\intercal G+\frac{1}{\epsilon}K^\intercal{\bf f}-K^\intercal \lambda^{(0)})
   \end{align}
   Using the first equation in \eqref{min11},i.e., $(\alpha B^\intercal B+\beta H^\intercal H)z^{(k+1)}=\alpha B^\intercal G-K^\intercal \lambda^{(k+1)}$ to replace $\alpha B^\intercal G$ in \eqref{eq12}, we have
 \begin{align}\label{algorithm}
 (\alpha B^\intercal B+\beta H^\intercal H+\frac{1}{\epsilon}K^\intercal K)z^{(k+1)}=(\alpha B^\intercal B+\beta H^\intercal H)z^{(k)}+\frac{1}{\epsilon}K^\intercal{\bf f}.
 \end{align}
We get the minimizer using \eqref{initial} and \eqref{algorithm}. 
These lead to Algorithm~\ref{alg1}. 

Next we show  the convergence of the above iterative algorithm. Since the minimization 
problem (\ref{min1}) is convex over a convex feasible set, we know that the minimization 
has a unique solution. We may assume that the linear system from Lagrange multiplier method 
has a  solution pair $(\lambda, z)$ with a unique solution $z$ if the size $|\triangle|$ 
of triangulation $\triangle$ 
is small enough and the spline space $S^r_D(\triangle)$ is dense enough in $H^2(\Omega)\cap 
H^1_0(\Omega)$. 
\begin{theorem}
Suppose that the matrices $K, H, B$ satisfy the following consistent condition: 
if $Kz=0, H z=0$, and $Bz=0$, one has $z=0.$  
 Then there exists a constant $\tilde{C}(\epsilon)$ depending on $\epsilon$ 
but independent of the iteration number $k$ such that
$$\|z^{(k+1)}-z\|\le \|\tilde{K}^{-1}\|\|K^\intercal \|\Big{(}\frac{\tilde{C} \epsilon}{1+\tilde{C} \epsilon}\Big{)}^{k+1}$$
for $k\geq 1,$ 
where $\tilde{C} =\|K^+\|^2 \|\alpha B^\intercal B+\beta H^\intercal H\|$  
and $K^+$ stands for the pseudo inverse of $K$ and 
$\tilde{K}=\alpha B^\intercal B+\beta H^\intercal H+\frac{1}{\epsilon}K^\intercal K$.  
\end{theorem}
\begin{proof}
First, we show that $\tilde{K}$ is invertible for $\alpha, \beta >0$. 
If $\tilde{K}z=0,$ we have 
 $$c^\intercal \tilde{K} c=\alpha \|Bc\|^2+\beta \|Hc\|^2+\frac{1}{\epsilon}\|Kc\|^2=0$$
which implies that $Kz=0, Hz=0, Bz=0$. By the assumption, $z=0.$ Thus, $\tilde{K}$ is invertible and hence the sequence $\{z^{(k)}\}$ is well-defined. Let $\tilde{C}_1= \|\tilde{K}\|$ which depends on $\epsilon.$\\
From \eqref{min10} and \eqref{eq12} ,
\begin{align*}
  \tilde{K} z^{(k+1)}&=\alpha B^\intercal G+\frac{1}{\epsilon}K^\intercal{\bf f}-K^\intercal \lambda^{(k)}\\
   \tilde{K} z&=\alpha B^\intercal G+\frac{1}{\epsilon}K^\intercal{\bf f}-K^\intercal \lambda .
\end{align*}
Hence, we have
\begin{align}\label{eq14}
z^{(k+1)}-z= \tilde{K}^{-1} K^\intercal (\lambda-\lambda^{(k)}).
\end{align}
By using \eqref{min11} and \eqref{eq12}, we get
\begin{align*}
-\epsilon (\lambda^{(k+1)}-\lambda)=-\epsilon (\lambda^{(k)}-\lambda)-\textbf{f}-K z^{(k+1)}
\end{align*} 
and
\begin{align*}
z^{(k+1)}=\tilde{K}^{-1}(\alpha B^\intercal G+\frac{1}{\epsilon} K^\intercal \textbf{f}- K^\intercal \lambda^{(k)}).
\end{align*} 
It follows that
\begin{align*}
-\epsilon (\lambda^{(k+1)}-\lambda)&=-\epsilon (\lambda^{(k)}-\lambda)-\textbf{f}-K z^{(k+1)}\\
&=-\epsilon (\lambda^{(k)}-\lambda)-\textbf{f}-K \tilde{K}^{-1}(\alpha B^\intercal G+\frac{1}{\epsilon} K^\intercal \textbf{f}- K^\intercal \lambda^{(k)})\\
&=-\epsilon (\lambda^{(k)}-\lambda)-\textbf{f}-K \tilde{K}^{-1}(\tilde{K}z+K^\intercal \lambda- K^\intercal \lambda^{(k)})\\
&=-\epsilon (\lambda^{(k)}-\lambda)-\textbf{f}-Kz-K  \tilde{K}^{-1} K^\intercal (\lambda- \lambda^{(k)}).
\end{align*}
As a result, we get 
\begin{align}\label{eq15}
 (\lambda^{(k+1)}-\lambda)= (\lambda^{(k)}-\lambda) (I-\frac{1}{\epsilon} K  \tilde{K}^{-1} K^\intercal ).
\end{align}
In order to show the next step, we use Lemma 7 in \cite{ALW06}, i.e., $\mathbb{R}^m=\text{Ker}(K^\intercal)\oplus \text{Im}(K)$ where $\text{Ker}(K^\intercal)$ is the kernel of $K^\intercal$.
Assume that $\lambda \in \text{Im}(K)$. By the second equation in \eqref{min11} that 
$$K(z^{(k+1)}-z)=\epsilon (\lambda^{(k)}-\lambda^{(k+1)}).$$
That is, $\lambda^{(k)}-\lambda^{(k+1)}$ is in the $\text{Im}(K)$ and therefore $$\lambda^{(k)}-\lambda=\sum_{j=1}^k (\lambda^{(j)}-\lambda^{(j-1)}) +(\lambda^{(0)}-\lambda),$$
we have $\lambda^{(k)}-\lambda \in \text{Im}(K)$ for each $k.$ From \eqref{eq15}, we need to estimate the norm of $I-\frac{1}{\epsilon} K  \tilde{K}^{-1} K^\intercal$ restricted to $\text{Im}(K)$ in order to estimate the norm of $\lambda^{(k+1)}-\lambda.$ We write $\|I-\frac{1}{\epsilon} K  \tilde{K}^{-1} K^\intercal\|$ for $\|(I-\frac{1}{\epsilon} K  \tilde{K}^{-1} K^\intercal )|_{\text{Im}(K)}\|$ and we have:
$$\|\lambda^{(k+1)}-\lambda\|\le \|I-\frac{1}{\epsilon} K  \tilde{K}^{-1} K^\intercal\| \|\lambda^{(k)}-\lambda\|.$$

We claim that 
$$ \|I-\frac{1}{\epsilon} K  \tilde{K}^{-1} K^\intercal\|  \le \frac{\tilde{C}_2 \epsilon}{1+\tilde{C}_2 \epsilon},$$
for some constant $\tilde{C}_2>0.$ Indeed, by the Rayleigh-Ritz quotient, we have
$$\|\lambda^{(k+1)}-\lambda\|\le \|I-\frac{1}{\epsilon} K  \tilde{K}^{-1} K^\intercal\|=\max_{0\neq v\in \text{Im}(K)}  (1-\frac{1}{\epsilon} \frac{v^\intercal K  \tilde{K}^{-1} K^\intercal v}{v^\intercal v}).$$
By using a technique from \cite{ALW06}, we can get 
$$\frac{1}{\epsilon} \frac{v^\intercal K  \tilde{K}^{-1} K^\intercal v}{v^\intercal v}>\frac{1}{1+\tilde{C}_2 \epsilon}, ~~\forall v\in \text{Im}(K)$$
where $\tilde{C}_2 =\|K^+\|^2 \|\alpha B^\intercal B+\beta H^\intercal H\|.$
It follows that 
$$ \|I-\frac{1}{\epsilon} K  \tilde{K}^{-1} K^\intercal\| \le 1-\frac{1}{1+\tilde{C}_2 \epsilon}=\frac{\tilde{C}_2 \epsilon}{1+\tilde{C}_2 \epsilon}.$$
As a results, we obtain
$$\|\lambda^{(k+1)}-\lambda\|\le \frac{\tilde{C}_2 \epsilon}{1+\tilde{C}_2 \epsilon}\|\lambda^{(k)}-\lambda\|$$
and from \eqref{eq14}
$$\|z^{(k+1)}-z\|\le \|\tilde{K}^{-1}\|\|K^\intercal \|\Big{(}\frac{\tilde{C}_2 \epsilon}{1+\tilde{C}_2 \epsilon}\Big{)}^{k+1}\|\lambda^{(0)}-\lambda\|.$$
\end{proof}
 

\begin{thebibliography}{66}
\bibitem{AL05}
Awanou, G. and Lai, M. -J., On Convergence Rate of the Augmented Lagrangian Algorithm for Nonsymmetric Saddle Point Problems, Journal of Applied Numerical Mathematics, vol. 54 (2005) pp. 122--134. 
\bibitem {ALW06}
G. Awanou, M. -J. Lai, and P. Wenston, The multivariate spline method for scattered data
fitting and numerical solution of partial differential equations. In Wavelets and splines: Athens
2005,  pages 24--74. Nashboro Press, Brentwood, TN, 2006. 

\bibitem{B11}
 H. Brezis, 
Functional analysis, Sobolev spaces and partial differential equations, Springer, 2011.

 
\bibitem{E98}
L. Evens,  Partial Differential Equation. American Mathematical Society, Providence (1998)


\bibitem{GL20}
F. Gao and  M. -J. Lai, 
A new $H^2$ regularity condition of the solution to Dirichlet problem of the Poisson equation and 
its applications, Acta Mathematica Sinica, vol. 36 (2020) pp. 21--39. 

\bibitem{G85}
P. Grisvard, Ellitpic Problems in Nonsmooth Domains, Pitman, 1985. 


\bibitem{HL07}
 X.-L. Hu, D.-F. Han, and M.-J Lai, Bivariate Splines of Various Degrees for Numerical Solution 
of Partial Differential Equations, SIAM J. Sci. Comput., 29(3), 1338--1354. (2007) 



\bibitem{L89}
M. -J. Lai, {\sl On Construction of Bivariate and Trivariate Vertex Splines
on Arbitrary Mixed Grid Partitions}, Dissertation, Texas A\&M University, 1989.

\bibitem{LM17}
M. -J. Lai and Mersmann, C., Adaptive Triangulation Methods for Bivariate Spline Solutions of PDEs, Approximation Theory XV: San Antonio, 2016, edited by G. Fasshauer and L. L. Schumaker, Springer Verlag, (2017),  pp. 155--175.

\bibitem{LS07}
M. -J. Lai and L. L. Schumaker, {\sl Spline Functions over Triangulations}, Cambridge University 
Press, 2007.

\bibitem{LS07b}
M. -J. Lai and L. L. Schumaker, Trivariate $C^r$ polynomial macro-elements. Constr. Approx. 26 (2007), no. 1, 11--28. 


\bibitem{LW17}
M. -J. Lai and Wang, C. M., A bivariate spline method for 2nd order elliptic equations in non-divergence form, 
Journal of Scientific Computing , (2018) pp. 803--829. 

\bibitem{LW21}
M. -J. Lai and Y. Wang, Sparse Solutions to Underdetermined Linear Systems, Publication, 
Philadelphia (2021).

\bibitem{LW04}
M. -J. Lai and Wenston, P., 
Bivariate Splines for Fluid Flows, Computers and Fluids, vol. 33 (2004) pp. 1047--1073. 

\bibitem{L22}
J. Lee, A Multivariate Spline Method for Numerical Solution of Partial Differential Equations, 
Dissertation (under preparation), University of Georgia, 2023. 

\bibitem{MY17}
 L. Mu and X. Ye, A simple finite element method for non-divergence form elliptic equations, 
 International Journal of Numerical Analysis and Modeling 14(2)(2017), pp. 306--311.


\bibitem{S15}
L. L. Schumaker,  Spline Functions: Computational Methods. SIAM Publication, Philadelphia (2015).

\bibitem{S19}
L. L. Schumaker,  Solving elliptic PDE's on domains with curved boundaries with 
an immersed penalized boundary method. J. Sci. Comput. 80 (2019), no. 3, 1369--1394.
 
\bibitem{SS13}
I. Smears, and E. S\"uli, Discontinuous Galerkin finite element approximation of 
nondivergence form elliptic
equations with Cordes coefficients. SIAM J. Numer. Anal. 51(4), 2088--2106 (2013).


\bibitem{WW19}
C. Wang, J. Wang, A primal dual weak Galerkin finite element method for 
second order elliptic equations in non-divergence form, Math. Comp., 2019. 
\end{thebibliography}
\end{document}